\documentclass[11pt,a4paper,english]{article}

\newcommand{\beq}{\begin{equation}}
\newcommand{\eeq}{\end{equation}}

\usepackage[latin1]{inputenc}
\usepackage{babel}
\usepackage[centertags]{amsmath}
\usepackage{amsfonts}
\usepackage{amssymb}
\usepackage{color}
\usepackage{amsthm}
\usepackage{newlfont}
\usepackage{graphicx}
\usepackage{dsfont}

\DeclareMathOperator*{\essinf}{ess\,inf}

\DeclareMathOperator*{\Var}{Var}

\DeclareMathOperator*{\cl}{cl}

\def \ep{\hbox{ }\hfill$\Box$}

\newtheorem{theorem}{Theorem}[section]
\newtheorem{lemma}[theorem]{Lemma}
\newtheorem{corollary}[theorem]{Corollary}
\newtheorem{proposition}[theorem]{Proposition}

\newtheorem{remark}[theorem]{Remark}
\newtheorem{Assumptions}[theorem]{Assumption}

\newcommand{\PP}{ \mathbb{P}}

\makeatletter  
\@addtoreset{equation}{section}
\def\theequation{\arabic{section}.\arabic{equation}}
\makeatother  
\begin{document} 
\title{\textbf{On the Optimal Management of Public Debt:\\ a Singular Stochastic Control Problem}\footnote{Financial support by the German Research Foundation
(DFG) through the Collaborative Research Centre 1283 ``Taming uncertainty and profiting from randomness and low regularity in analysis, stochastics and their applications'' is gratefully acknowledged.}}
\author{Giorgio Ferrari\thanks{Center for Mathematical Economics, Bielefeld University, Germany; \texttt{giorgio.ferrari@uni-bielefeld.de}}}
\date{\today}
\maketitle

\vspace{0.5cm}

{\textbf{Abstract.}} Consider the problem of a government that wants to reduce the debt-to-GDP (gross domestic product) ratio of a country. The government aims at choosing a debt reduction policy which minimises the total expected cost of having debt, plus the total expected cost of interventions on the debt ratio. We model this problem as a singular stochastic control problem over an infinite time-horizon. In a general not necessarily Markovian framework, we first show by probabilistic arguments that the optimal debt reduction policy can be expressed in terms of the optimal stopping rule of an auxiliary optimal stopping problem. We then exploit such link to characterise the optimal control in a two-dimensional Markovian setting in which the state variables are the level of the debt-to-GDP ratio and the current inflation rate of the country. The latter follows uncontrolled Ornstein-Uhlenbeck dynamics and affects the growth rate of the debt ratio. We show that it is optimal for the government to adopt a policy that keeps the debt-to-GDP ratio under an inflation-dependent ceiling. This curve is given in terms of the solution of a nonlinear integral equation arising in the study of a fully two-dimensional optimal stopping problem.  
\smallskip


\smallskip

{\textbf{MSC2010 subject classification}}: 93E20, 60G40, 91B64, 45B05, 60J60.

\smallskip


\section{Introduction}
\label{introduction}

Controlling the debt-to-GDP ratio (also called the ``debt ratio") and keeping it below some desirable level is of fundamental importance for all countries. It has been shown by different authors by means of different statistical and methodological approaches, that high government debt has a negative effect on the long-term economic growth. The usual outcome is that when government debt grows, private investment shrinks, and future growth and future wages lower (see, e.g., \cite{WK15}). In \cite{RRR12} it is shown that high government debt hurts growth even in the absence of a crisis. This negative effect on economic growth from high debt levels has been observed  in 18 different advanced economies (see \cite{CMZ11}). 

In this paper we propose a continuous-time stochastic model for the control of the debt-to-GDP ratio. The problem we have in mind is that of a government aiming to answer the question: \emph{How much is too much?}\footnote{cf.\ \emph{The Economist}, June 3rd 2015.} Following classical macroeconomic theory (see, e.g., \cite{BlanchardFischer}), in any given period the debt ratio stock grows by the existing debt stock multiplied by the difference between real interest rate and GDP growth, less the primary budget balance\footnote{A government budget balance is the overall difference between government revenues and spending. The government budget balance is further split into the primary balance and the structural balance. The primary budget balance equals the government budget balance before interest payments.}. We assume that the government can reduce the level of the debt-to-GDP ratio by adjusting the primary budget balance, e.g.\ through fiscal  interventions like raising taxes or reducing expenses. We therefore interpret the cumulative interventions on the debt ratio as the government's control variable, and we model it as a nonnegative and nondecreasing stochastic process.

Uncertainty in our model comes through the GDP growth rate of the country, and its inflation and nominal interest rate which, by Fisher law \cite{Fisher}, directly affect the growth rate of the debt ratio. We first assume that they are general not necessarily Markovian stochastic processes whose dynamics is not under government control. Indeed, the level of these macroeconomic variables is usually regulated by an autonomous Central Bank, whose action, however, is not modelled in this paper (see, e.g., \cite{CH00} and \cite{HHSZ} for problems related to the optimal control of inflation).

Since high debt-to-GDP ratios can constrain economic growth making it more difficult to break the burden of the debt, we assume that debt ratio generates an instantaneous cost/penalty. This is a general nonnegative convex function of the debt ratio level, that the government would like to keep as close to zero as possible. However, at any time the government decides to intervene in order to reduce the level of the debt ratio, it incurs a cost which is proportional to the amount of the debt reduction. The government thus aims at choosing a (cumulative) debt reduction policy minimising the sum of the total expected cost of having debt, and of the total expected cost of interventions on the debt ratio. 

We model the government's debt management problem as a singular stochastic control problem (see \cite{Shreve} for an introduction), i.e.\ a problem in which the control processes may be singular, as functions of time, with respect to the Lebesgue measure. In our general not necessarily Markovian setting we are able to link the solution of the control problem to that of a suitable optimal stopping problem. In line with \cite{BK}, \cite{DeAFeFe} and \cite{KaratzasShreve84}, among others, we show that the optimal stopping time $\tau^*$ of an auxiliary optimal stopping problem defines the optimal control $\nu^*$ of the original singular stochastic control problem: $\nu^*$ is directly related to the generalised inverse of $\tau^*$. The proof of this result is fully probabilistic, and no regularity of the value functions of the control and optimal stopping problems is needed. Indeed, it relies on a change of variable formula for Lebesgue-Stieltjes integrals.

The link to optimal stopping is then exploited to characterise the optimal debt reduction policy in a Markovian two-dimensional setting. Here the state variables are the levels of the debt ratio and of the inflation rate, while the nominal interest rate and the GDP growth rate are kept constant (see also Remark \ref{rem:dimensionality} below). Assuming that the inflation rate $Y$ evolves according to a one-dimensional Ornstein-Uhlenbeck process, guided by our previous findings we relate the singular stochastic control problem to a suitable fully two-dimensional optimal stopping problem. In the optimal stopping problem the state variables are the inflation rate process $Y$ and its time-integral $Z$. It is well known that $(Z,Y)$ is a time-homogeneous, strong Markov process, whose first component is of bounded-variation (being a time integral). It therefore turns out that the free-boundary formulation of the optimal stopping problem involves a second-order linear partial differential equation of (local) parabolic type. Relying on almost exclusively probabilistic arguments, we show that the optimal stopping rule is triggered by a curve $\hat{y}$ uniquely solving (within a certain functional class) a nonlinear integral equation. Such characterisation of $\hat{y}$ finally gives rise to a complete caracterisation of the optimal debt reduction policy. Indeed, at each time $t$, the latter prescribes to keep the debt ratio $X_t$ below an inflation-dependent level $b(Y_t)$, with the curve $b$ being related to the (generalised) inverse of the free boundary $\hat{y}$. The explicit form of the optimal debt reduction strategy allows us also to obtain interesting economic conclusions (see Section \ref{sec:optimalsol}).

It is worth noticing that the number of papers characterising the optimal policy in multi-dimensional singular stochastic control problems is still limited (see \cite{DeAFeFe} for a recent contribution). In some early papers (see, e.g., \cite{CH94}, \cite{MenaldiTaksar}, \cite{SonerShreve} and \cite{CMW}) fine analytical methods based on the dynamic programming principle and the theory of variational inequalities are employed to study the regularity of the value function of multi-dimensional singular stochastic control problems, and of the related free boundary. However, a characterisation of the latter is not obtained in these works. More recently, assuming that the control acts only in one dimension, two-dimensional degenerate singular stochastic control problems motivated, e.g., by questions of optimal (ir)reversible investment under uncertainty have been solved via different methods: the related free boundary has been characterised through a suitable equation, and the value function has been explicitly determined. We refer to \cite{FedericoPham}, \cite{Ferrari}, \cite{GT} and \cite{MZ}, among others. However, the setting of those papers is different with respect to ours. In fact there the dynamics of a purely controlled state, modelling the production capacity of a firm, is independent of an uncontrolled diffusive process, representing the demand of a produced good or other factors influencing the company's running profit. As a consequence, in those papers the linear part of the dynamic programming equation for the problem's value function takes the form of an ODE (rather than a PDE), so that in the inaction region the value function depends only parametrically on the variable associated to the purely controlled state. In the Markovian formulation of our problem we have instead a coupling between the components of the state process (see eqs.\ \eqref{Xinter} and \eqref{dynY} below), and this makes our framework fully two-dimensional. However, by exploiting a nontrivial connection to optimal stopping, and performing an almost exclusively probabilistic analysis, we still obtain a complete characterisation of the control problem's value function and of the optimal strategy. To the best of our knowledge, this is a novelty in the literature on singular stochastic control. We also believe that the detailed analysis of the free boundary that we obtain in this paper contributes to the literature on optimal stopping, since examples of solvable two-dimensional optimal stopping problems are quite rare (see \cite{Christensenetal}, Section 3 in \cite{DeAFeFe}, and \cite{PeskirJohnson} for recent contributions).

As we comment on in Section \ref{sec:relatedprobl}, our debt management problem shares a similar mathematical structure with other optimisation problems arising in economic theory/mathematical finance, which might then be tackled through the approach and techniques of this paper. This is the case of problems of capacity expansion \cite{DP} with (a possibly stochastic) depreciation in the capital stock (see \cite{RS}, among others), and problems of optimal consumption under intertemporal preferences \`a la Hindy-Huang-Kreps (cf.\ \cite{BankRiedel1}, \cite{HH} and references therein). 

Our interest in stochastic control methods for public debt management started reading the recent \cite{CadAgui} (see also \cite{CadAguiOR}), which,
to the best of our knowledge, is the only other paper dealing with a mathematical rigorous analysis of the debt reduction problem. In \cite{CadAgui} the debt ratio evolves according to a linearly controlled one-dimensional geometric Brownian motion, and the government aims at minimising the total expected costs arising from having debt and intervening on it. Although the government cost functional we consider in this paper is similar to the one in \cite{CadAgui}, the Markovian formulation of our singular control problem is fully two-dimensional, whereas that of \cite{CadAgui} is one-dimensional. This implies that our optimal debt ceiling is a curve, whereas that of \cite{CadAgui} is a constant.

The rest of the paper is organised as follows. In Section \ref{problem} we set up the model and introduce the control problem. Section \ref{sec:verification-general} is devoted to establishing, in a general not necessarily Markovian setting, the link between the singular stochastic control problem and an auxiliary optimal stopping problem. In Section \ref{Markovsetting} we restrict our attention to a Markovian framework so to obtain a complete characterisation of the optimal debt reduction policy (Subsections \ref{sec:OS} and \ref{sec:OC}). Some economic conclusions and comments on related problems can be found in Section \ref{sec:optimalsol}, whereas Appendix \ref{someproofs} collects some proofs, and Appendix \ref{auxiliaryres} contains some auxiliary results.


\section{The General Model and the Control Problem}
\label{problem}

Within a not necessarily Markovian setting, in this section we provide a general formulation of our model and of the control problem. 

Let $X_t$ be the level of public debt-to-gross domestic product (GDP) ratio at time $t \geq 0$; that is, 
$$
X_t:=\frac{\text{gross public debt at time $t$}}{\text{GDP at time $t$}}.
$$
According to classical macroeconomic theory (see, e.g., \cite{BlanchardFischer}), in any given period the debt stock grows by the existing debt stock multiplied by the difference between real interest rate and GDP growth, less the primary budget balance. By Fisher law \cite{Fisher} the real interest rate is given by the difference of the nominal interest rate and the inflation rate. In reality these variables are all time-dependent, stochastic and interrelated. We therefore assume that they are described by not necessarily Markovian real-valued processes on a given complete filtered probability space $(\Omega, \mathcal{F}, \mathbb{F}, \mathbb{P})$, with filtration $\mathbb{F}:=\{\mathcal{F}_t, t\geq 0\}$ satisfying the usual conditions.  $\mathbb{F}$ is the flow of information available to the government. We denote by $g:=\{g_t, t \geq 0\}$ the real-valued GDP growth rate, by $\delta:=\{\delta_t, t \geq 0\}$ the nonnegative nominal interest rate, and by $Y:=\{Y_t, t\geq 0\}$ the real-valued inflation rate. We take all of them $\mathbb{F}$-progressively measurable and such that the process $\beta$, defined by 
\beq
\label{beta}
\beta_t:=\delta_t - Y_t - g_t, \quad t\geq 0,
\eeq
satisfies $\int_0^t |\beta_u| du < \infty$ a.s.\ for any $t\geq 0$.

The dynamics of $X$ then takes the form
\beq
\label{dynX}
dX_t = (\delta_t - Y_t - g_t)X_t dt - d\nu_t, \quad t \geq 0, \qquad X_{0-}=x>0,
\eeq
where $\nu_t$ is the cumulative primary balance up to time $t$. The primary balance is the variable that the government can control, e.g.\ through fiscal interventions, in order to reduce the debt ratio level. We observe from \eqref{dynX} that the nominal interest rate, the inflation rate and the GDP growth rate directly affect the dynamics of $X$. However, they are not under government control. Inflation rate and nominal interest rate are indeed usually regulated by an autonomous Central Bank, whereas the growth rate typically depends on the business cycles of the economy.

Equation \eqref{dynX} can be explicitly solved yielding for any $t\geq 0$ and $x \in (0,\infty)$
\beq
\label{syst:X}
X^{\nu}_t=e^{\int_0^t \beta_s ds}\bigg[ x - \int_0^t e^{-\int_0^s \beta_u du} d\nu_s\bigg].
\eeq

In \eqref{syst:X} above and in the rest of this section we shall write $X^{\nu}$ to account for the dependence of $X$ on the control policy $\nu$. Moreover, we will write $X^{x,\nu}$ when there will be the need to stress the dependence of $X^{\nu}$ on the initial level $x \in (0,\infty)$ as well. The admissible policies that the government can employ to decrease the level of debt ratio are drawn from the set
\begin{eqnarray}
\label{admissiblecontrols-general}
\mathcal{A} \hspace{-0.2cm}&:=&\hspace{-0.2cm} \{\nu:\Omega \times \mathbb{R}_{+} \mapsto  \mathbb{R}_{+}, ({\nu_{t}(\omega) := \nu(\omega,t)})_{t \geq 0}\mbox{ is nondecreasing,\,\,right-continuous,} \nonumber \\
&& \hspace{1.7cm} \mbox{$\mathbb{F}$-adapted,\,\,such that} \,\,X^{\nu}_t\geq 0\,\,\,\PP-\mbox{a.s.}\,\,\forall \; t \geq 0\}. 
\end{eqnarray}
In the following we set $\nu_{0-}=0$ a.s.\ for any $\nu \in \mathcal{A}$.
Notice that in \eqref{admissiblecontrols-general} we do not allow for policies that let the debt ratio become negative, i.e.\ that make the government a net lender. This is somehow a realistic requirement, as a situation with negative debt is less relevant in real world economies.

The government aims at reducing the level of debt ratio. Having a level of debt ratio $X_t$ at time $t\geq 0$ the government incurs an instantaneous random cost $h(\omega, t, X_t(\omega))$, $\omega \in \Omega$. This may be interpreted as a measure of the resulting losses for the country due to the debt, as, e.g., a tendency to suffer low subsequent growth (see \cite{CMZ11}, \cite{RRR12}, \cite{WK15}, among others, for empirical studies). The randomness of the cost function captures the possible dependency of the country's losses on other exogenous factors, as the overall economic situation. Regarding such instantaneous cost we make the following assumption.
\begin{Assumptions}
\label{ass:h}
$h: \Omega \times \mathbb{R}_+ \times \mathbb{R} \mapsto \mathbb{R}_+$ is such that
\begin{itemize}
\item[(i)] $(\omega,t) \mapsto h(\omega,t,x)$ is $\mathbb{F}$-progressively measurable for any $x\in \mathbb{R}$;
\item[(ii)] $x \mapsto h(\omega,t,x)$ is strictly convex, continuously differentiable, nondecreasing on $\mathbb{R}_+$, and such that $h(\omega,t,0)=0$ for any $(\omega,t) \in \Omega \times \mathbb{R}_+$;
\item[(iii)] $$\mathbb{E}\bigg[\int_0^{\infty} h(\omega,t,X^{0}_t(\omega)) dt\bigg] + \mathbb{E}\bigg[\int_0^{\infty} X^{0}_t(\omega) h_x(\omega,t,X^{0}_t(\omega)) dt\bigg] < \infty.$$
\end{itemize}
\end{Assumptions}
In the previous assumption, and in the following, $X^0$ denotes the solution of \eqref{dynX} when we take $\nu\equiv 0$. For notational simplicity, from now on we will drop the $\omega$-dependence in the cost function and in all the stochastic processes when it is not necessarily needed.

\begin{remark}
\label{rem:h}
Notice that the requirement $h(t,0)=0$ is without loss of generality, since if $h(t,0) = f(t) > 0$ then one can always set $\widehat{h}(t,x):= h(t,x) - f(t)$ and write $h(t,x) = \widehat{h}(t,x) + f(t)$, so that the optimisation problem (cf.\ \eqref{eq:valueOC-general} below) remains unchanged up to an additive constant.
\end{remark}

Whenever the government decides to reduce the level of the debt ratio, it incurs an intervention cost that is proportional to the amount of the debt reduction. Fiscal adjustments as raising taxes or reducing expenses may generate such a cost. The marginal cost of intervention is described by a positive stochastic process $\kappa$ satisfying the following requirements.
\begin{Assumptions}
\label{ass:kappa}
$\kappa: \Omega \times \mathbb{R}_+ \mapsto (0,\infty)$ is $\mathbb{F}$-adapted with continuous paths, and it is such that $\{X^0_t \kappa_t, t \geq 0\}$ is of \emph{class (D)}.\footnote{A process $\xi$ is of \textsl{class (D)} if $\{\xi_\tau,\,\,\tau<\infty\,\,\text{a stopping time}\}$ defines a uniformly integrable family of random variables on $(\Omega,\mathcal{F},\mathbb{P})$.
We refer the reader to Definition 1.6 at p.\ 124 of \cite{RY} for further details.}
\end{Assumptions}

The government's goal is to choose a policy $\nu^* \in \mathcal{A}$ minimising the total expected cost
\beq
\label{eq:J-general}
\mathcal{J}_{x}(\nu):=\mathbb{E}\bigg[\int_0^{\infty} h(t, X^{x,\nu}_t) dt + \int_0^{\infty} \kappa_t d\nu_t\bigg], \quad x \in (0,\infty).
\eeq
For any $\nu \in \mathcal{A}$, $\mathcal{J}_{x}(\nu)$ is well defined but possibly infinite. For $x \in (0,\infty)$ the minimal cost is therefore
\beq
\label{eq:valueOC-general}
V(x):=\inf_{\nu \in \mathcal{A}}\mathcal{J}_{x}(\nu).
\eeq
Notice that $V$ is finite since for any $x>0$ one has $V(x) \leq \mathcal{J}_{x}(0) < \infty$, where the last inequality is due to Assumption \ref{ass:h}-(iii).

Problem \eqref{eq:valueOC-general} takes the form of a singular stochastic control problem, i.e.\ of a problem in which control processes may be singular with respect to the Lebesgue measure, as functions of time (see \cite{Shreve} for an introduction; \cite{Karatzas81} and \cite{KaratzasShreve84} as classical references in Markovian settings; \cite{BK}, and \cite{CFR}, among others, for studies in not necessarily Markovian frameworks).

\begin{remark}
\label{rem:relatedprobl}
\begin{enumerate}\hspace{10cm}
\item Given the finiteness of $V$, there is no loss of generality in restricting the class of admissible controls to those belonging to $\mathcal{A}$ and such that $\mathcal{J}_x(\nu)<\infty$. Then, by strict convexity of $h(t,\cdot)$, and the affine nature of $X^{x,\nu}$ in $\mathcal{\nu}$, it follows that $\mathcal{J}_{x}(\,\cdot\,)$ is strictly convex. Therefore the solution of the government's debt management problem, if it exists, is unique.
\item Problem \eqref{eq:valueOC-general} shares a common mathematical structure with problems of optimal consumption under Hindy-Huang-Kreps (HHK) preferences (cf.\ \cite{BankRiedel1}, \cite{HH} and references therein), and problems of stochastic irreversible investment (see Chapter 11 in \cite{DP} for a review) with (stochastic) depreciation of the capital stock (see \cite{RS}, among others, for a problem of irreversible investment with constant depreciation of the capital stock). More details on these classes of optimisation problems and on their relation to problem \eqref{eq:valueOC-general} can be found in Section \ref{sec:relatedprobl}.
\end{enumerate}
\end{remark}


\section{A Probabilistic Verification Theorem}
\label{sec:verification-general}

In this section we provide a general verification theorem for the singular stochastic control problem \eqref{eq:valueOC-general}. We show that its solution is related to that of an auxiliary optimal stopping problem formulated in terms of the marginal cost of having debt, $h_x$, and the marginal cost of intervention on the debt ratio, $\kappa$. In the optimal stopping problem the government decides \emph{when} to reduce the debt ratio level by one additional unit in order to minimise the associated total expected marginal cost.

Recall \eqref{beta}, for any given $z \in \mathbb{R}$ set
\beq
\label{greekzeta}
\zeta_t(\omega, z):=\int_0^{t} e^{z + \int_0^s \beta_u(\omega) ds}\, h_x(\omega, t, e^{z + \int_0^s \beta_u(\omega) du}) ds + \kappa_{t}(\omega)\, e^{z + \int_0^{t} \beta_u(\omega) du}, \quad (\omega,t)\in \Omega \times \mathbb{R}_+,
\eeq
and introduce the optimal stopping problem 
\beq
\label{valueOS-general}
\widetilde{U}_t(z):=\essinf_{\tau \geq t}\mathbb{E}\big[\zeta_{\tau}(z)\,\big|\,\mathcal{F}_t\big], \qquad t\geq 0,
\eeq
where the optimisation is taken over all $\mathbb{F}$-stopping times $\tau \geq t$. In \eqref{valueOS-general} we use the convention (cf.\ \cite{Shir}, Ch.\ 3)
$$\kappa_{\tau}\, e^{z + \int_0^{\tau} \beta_u du}:=\liminf_{t\uparrow \infty}\kappa_{t}\, e^{z + \int_0^{t} \beta_u du} \quad \text{on}\quad \{\tau = \infty\}.$$

Denoting by $U_t(z)$ a c\`adl\`ag modification of $\widetilde{U}_t(z)$, under Assumptions \ref{ass:h} and \ref{ass:kappa} the stopping time 
\begin{equation}
\label{taustargeneral}
\tau^*(z):=\inf\{t \geq 0: U_t(z) \geq \zeta_t(z)\}, \quad z \in \mathbb{R},
\end{equation}
is optimal for problem \eqref{valueOS-general} by Theorem D.12 in Appendix D of \cite{KS-MF}. Notice that since $h_x(t, \cdot)$ is a.s.\ increasing, then $z \mapsto \tau^*(z)$ is a.s.\ decreasing (see Lemma \ref{inverse} in Appendix \ref{auxiliaryres} for a short proof). Such monotonicity of $\tau^*(\,\cdot\,)$ will be important in the following as we will need to consider the generalised inverse of $\tau^*(\ln(\,\cdot\,))$.

Letting $U(z):=U_0(z)$, for any $x>0$ define
\beq
\label{hatV}
\widehat{V}(x):= \int_{-\infty}^{\ln(x)}U(z)dz.
\eeq
Since we can also write $\widehat{V}(x)=\int_0^x \frac{1}{z}U(\ln(z))dz$, $x>0$, where 
\begin{equation}
\label{def:1suxu-general}
\frac{1}{x}U(\ln(x)) = \inf_{\tau \geq 0}\mathbb{E}\bigg[\int_0^{\tau} e^{\int_0^t \beta_s ds}\, h_x\big(t, x e^{\int_0^t \beta_s ds}\big) dt  + \kappa_{\tau}\, e^{\int_0^\tau \beta_s ds}\bigg],
\end{equation}
for any $x>0$ one has that $0 \leq \widehat{V}(x) \leq \mathbb{E}[\kappa_0]x< \infty$, where the last inequality is due to Assumption \ref{ass:kappa}.

Moreover, for any $x>0$ set $\sigma^*(x):= \tau^*(\ln(x))$, and introduce the nondecreasing, right-continuous process
\beq
\label{candidate-general}
\overline{\nu}^*_t := \sup\{\alpha \in [0,x]: \sigma^*((x-\alpha)\,+) \leq t\}, \quad t \geq 0, \qquad \overline{\nu}^*_{0-} =0,
\eeq
and then also the process
\beq  
\label{nustar-general}
\nu^*_t:=\int_{0}^t e^{\int_0^s \beta_u du} d\overline{\nu}^*_s, \quad t > 0, \qquad \nu^*_{0-}=0. 
\eeq
\noindent Notice that $\overline{\nu}^*_{\cdot}$ is the right-continuous inverse of $\sigma^*(\,\cdot\,)$.

The proof of the next theorem is fully probabilistic, and it relies on a change of variable formula already used in the context of singular control problems (see, e.g., \cite{BK}, \cite{DeAFeFe}, and \cite{KaratzasElKarouiSkorohod}).

\begin{theorem}
\label{teo:ver}
Let $\widehat{V}$ be as in \eqref{hatV} and $V$ as in \eqref{eq:valueOC-general}. Then one has $\widehat{V} = V$, and $\nu^*$ is the (unique) optimal control for problem \eqref{eq:valueOC-general}.
\end{theorem}

\begin{proof}
\emph{Step 1.} Let $x > 0$ be given and fixed. For $\nu \in \mathcal{A}$, introduce the process $\overline{\nu}$ such that
$\overline{\nu}_t:= \int_{0}^t e^{-\int_0^s \beta_u du} d{\nu}_s$, $t\geq 0$, and define its inverse (see, e.g., Chapter 0, Section 4 of \cite{RY}) by
\beq
\label{taunu}
\tau^{\overline{\nu}}(q):=\inf\{t\geq 0 \ | \ x - \overline{\nu}_t < e^q\}, \qquad q \leq \ln(x).
\eeq
The process $\tau^{\overline{\nu}}(q):=\{\tau^{\overline{\nu}}(q),\  q \leq \ln(x)\}$ has decreasing, left-continuous sample paths and hence it admits right-limits
\beq
\label{taunumeno}
\tau^{\overline{\nu}}_{+}(q):=\inf\{t\geq 0 \ | \ x - \overline{\nu}_t \leq e^q\}, \qquad q \leq \ln(x).
\eeq
Moreover, the set of points $q\in\mathbb{R}$ at which $\tau^{\overline{\nu}}(q)(\omega) \neq \tau^{\overline{\nu}}_{+}(q)(\omega)$ is a.s.\ countable for a.e.~$\omega\in\Omega$.

Since $\overline{\nu}$ is right-continuous and $\tau^{\overline{\nu}}(q)$ is the first entry time of an open set, it is an $(\mathcal{F}_{t+})$-stopping time for any given and fixed $q \leq \ln(x)$. However, $(\mathcal{F}_t)_{t\ge0}$ is right-continuous, hence $\tau^{\overline{\nu}}(q)$ is an $(\mathcal{F}_{t})$-stopping time. Moreover, $\tau^{\overline{\nu}}_{+}(q)$ is the first entry time of the right-continuous process $\overline{\nu}$ into a closed set and hence it is an $(\mathcal{F}_{t})$-stopping time as well for any $q \leq \ln(x)$. 

With regard to \eqref{hatV} we can then write
\begin{align}
\label{verifico-1}
& \widehat{V}(x) = \int_{-\infty}^{\ln(x)}U(q)dq \leq \int_{-\infty}^{\ln(x)} \mathbb{E}\bigg[\int_0^{\tau^{\overline{\nu}}(q)} e^{q + \int_0^t \beta_s ds}\, h_x(t, e^{q + \int_0^t \beta_s ds}) dt \nonumber \\
& \hspace{1.5cm} + \kappa_{\tau^{\overline{\nu}}(q)}\, e^{q + \int_0^{\tau^{\overline{\nu}}(q)} \beta_s ds}\bigg]\,dq,
\end{align}
and we now consider the two terms in the last integral above separately. 

We have
\begin{align}
\label{verifico-2}
& \int_{-\infty}^{\ln(x)} \mathbb{E}\Big[\kappa_{\tau^{\overline{\nu}}(q)}\, e^{ q + \int_0^{\tau^{\overline{\nu}}(q)} \beta_s ds}\Big]\,dq = \int_{0}^{x} \mathbb{E}\Big[\kappa_{\tau^{\overline{\nu}}(\ln(\xi))}\, e^{\int_0^{\tau^{\overline{\nu}}(\ln(\xi))} \beta_s ds}\Big]\,d\xi \nonumber \\
& = \mathbb{E}\bigg[\int_0^{\infty} \kappa_t\, e^{\int_0^{t} \beta_s ds}\,d\overline{\nu}_t\Big] = \mathbb{E}\bigg[\int_0^{\infty} \kappa_t\,d{\nu}_t\bigg], 
\end{align}
where the first step is due to the change of variable $\xi=e^q$, and the second step follows from Tonelli's theorem and the change of variable formula in Chapter 0, Proposition 4.9 of \cite{RY}, upon observing that $\tau^{\overline{\nu}}(\ln(\xi))=\inf\{t\geq 0 \ | \overline{\nu}_t > x - \xi\}$.

On the other hand, notice that for any $t\geq 0$ one has $t < \tau^{\overline{\nu}}(q)$ if and only if $\overline{\nu}_t < x - e^q$. Therefore, by applying Tonelli's theorem we have 
\begin{align}
\label{verifico-3}
& \int_{-\infty}^{\ln(x)} \mathbb{E}\bigg[\int_0^{\tau^{\overline{\nu}}(q)} e^{q + \int_0^t \beta_s ds}\, h_x(t, e^{q + \int_0^t \beta_s ds})dt \bigg] dq \nonumber \\
& = \mathbb{E}\bigg[\int_0^{\infty} \Big(\int_{-\infty}^{\ln(x)} e^{q + \int_0^t \beta_s ds}\, h_x(t, e^{q + \int_0^t \beta_s ds})\mathds{1}_{\{t < \tau^{\overline{\nu}}(q)\}} dq\Big) dt\bigg] \nonumber \\
& = \mathbb{E}\bigg[\int_0^{\infty} \Big(\int_{-\infty}^{\ln(x)} e^{q + \int_0^t \beta_s ds}\, h_x(t, e^{q + \int_0^t \beta_s ds}) \,\mathds{1}_{\{e^q < x - \overline{\nu}_t\}}\,dq\Big) dt\bigg] \\
& = \mathbb{E}\bigg[\int_0^{\infty} \Big(\int_0^{xe^{\int_0^t \beta_u du}} h_x(t,y) \,\mathds{1}_{\{y < e^{\int_0^t \beta_u du}\,(x - \overline{\nu}_t)\}}\,dy \Big) dt\bigg] = \mathbb{E}\bigg[\int_0^{\infty} h(t, X^{x,\nu}_t) dt\bigg], \nonumber 
\end{align}
where the last equality follows from \eqref{syst:X} and the fact that $h(t,0)=0$.

Combining \eqref{verifico-2} and \eqref{verifico-3} we thus have from \eqref{verifico-1} that $\widehat{V}(x) \leq \mathcal{J}_{x}(\nu)$. Hence, since $\nu$ was arbitrary
\beq
\label{ineq1}
\widehat{V}(x) \leq V(x), \quad x>0.
\eeq

\emph{Step 2.} To complete the proof we need to show the reverse inequality. To this end we preliminary notice that $\nu^* \in \mathcal{A}$ since it is nondecreasing, right-continuous and such that $X^{\nu^*}_t \geq 0$ a.s.\ for all $t\geq 0$. The latter property follows from \eqref{syst:X} upon observing that, for any $x>0$ and $t\geq 0$, one has by definition $\overline{\nu}^*_t \leq x$ a.s.

Also, the process $\nu^*$ is such that for any $t\geq 0$ and $q \leq \ln(x)$, $x>0$, we can write (cf.\ \eqref{candidate-general} and \eqref{taunumeno})
$$\tau^{\overline{\nu}^*}_+(q) \leq t \,\, \Longleftrightarrow \,\, \overline{\nu}^*_t \geq x - e^q\,\, \Longleftrightarrow \,\, \sigma^*(e^q) \leq t \,\, \Longleftrightarrow \,\, \tau^*(q) \leq t.$$

Then recalling that $\tau^{\overline{\nu}^{*}}_{+}(q)=\tau^{\overline{\nu}^{*}}(q)$ $\mathbb{P}$-a.s.~and for a.e.~$q\le \ln(x)$, we pick $\nu=\nu^*$ (equivalently, $\overline{\nu}=\overline{\nu}^*$), \eqref{verifico-1} becomes an equality, and \eqref{verifico-2} and \eqref{verifico-3} yield $\widehat{V}(x)=\mathcal{J}_{x}(\nu^*)$. That is, $\widehat{V}=V$ by \eqref{ineq1} and admissibility of $\nu^*$. Therefore $\nu^*$ is optimal (it is in fact the unique optimal control in the class of controls belonging to $\mathcal{A}$ and such that $\mathcal{J}_x(\nu) < \infty$ by strict convexity of $\mathcal{J}_x(\,\cdot\,)$).
\end{proof}

A direct byproduct of Theorem \ref{teo:ver} is the following.
\begin{corollary}
\label{cor:Vzv}
The identity $V_x(x)=\frac{1}{x}U(\ln(x))$ holds true for any $x>0$.
\end{corollary}

This result is consistent with the fact that problems of singular stochastic control with performance criterion which is either convex or concave with respect to the control variable are related to questions of optimal stopping. In particular, the derivative of the control problem's value function in the direction of the controlled state variable equals the value of an optimal stopping problem. We refer to \cite{BK} and \cite{KaratzasShreve84}, among others, as classical references.


\section{Characterisation of the Optimal Policy in a Markovian Setting}
\label{Markovsetting}

\subsection{Embedding the Control Problem in a Markovian Framework}

Theorem \ref{teo:ver} shows that in order to solve control problem \eqref{eq:valueOC-general} it suffices to solve optimal stopping problem \eqref{valueOS-general}. Indeed the right-continuous inverse of (an appropriate transformation of) the optimal stopping time of \eqref{valueOS-general} gives the optimal control for \eqref{eq:valueOC-general}. Moreover, a suitable integral of the value of the optimal stopping problem \eqref{valueOS-general} gives the value of the original singular control problem \eqref{eq:valueOC-general}.

In order to obtain a complete characterisation of the optimal debt reduction policy, we now restrict our attention to a Markovian framework, and we reduce the dimensionality of the problem. The following setting will be kept in the rest of the paper.

We take $(\Omega, \mathcal{F}, \mathbb{P})$ rich enough to accommodate a one-dimensional Brownian motion $W:=\{W_t,\,t \geq 0\}$, and we let $\mathbb{F}:=\{\mathcal{F}_t,\,t\geq 0\}$ be the Brownian filtration, as usual augmented by $\mathbb{P}$-null sets of $\mathcal{F}$. In this paper we focus on the role of the dynamic inflation rate in the public debt management problem, and in \eqref{dynX} we therefore keep the nominal interest rate and the GDP growth rate constant (see also Remark \ref{rem:dimensionality} below); that is $\delta_t \equiv \delta \geq 0$ and $g_t \equiv g \in \mathbb{R}$, so that (cf.\ \eqref{beta}) $\beta_t = \delta - g - Y_t$, $t\geq 0$, and
\beq
\label{Xinter}
dX^{\nu}_t = (\delta - g - Y_t)X^{\nu}_t dt - d\nu_t, \quad t \geq 0, \qquad X^{\nu}_{0-} = x>0.
\eeq
Supported by empirical evidence (see, e.g., \cite{LeeWu} and references therein), we model the inflation rate $Y$ as a stationary and mean-reverting process. In particular, it evolves as an Ornstein-Uhlenbeck process
\beq
\label{dynY}
dY_t = (a - \theta Y_t) dt + \sigma dW_t, \qquad Y_0=y \in \mathbb{R},
\eeq
where $a/\theta \in \mathbb{R}$ is the equilibrium level, $\sigma>0$ the volatility, and the parameter $\theta>0$ is the speed at which $Y$ asymptotically converges in average towards its equilibrium.

In the rest of this paper we will often write $(X^{x,y,\nu},Y^{y})$ to account for the dependence of $(X,Y)$ (cf.\ \eqref{Xinter} and \eqref{dynY}) on the initial levels $(x,y) \in (0,\infty) \times \mathbb{R}$ and on the control policy $\nu$. Also, we will stress the dependence of the set of admissible debt reduction policies \eqref{admissiblecontrols-general} on the initial levels $(x,y) \in (0,\infty) \times \mathbb{R}$ , and we will denote it by $\mathcal{A}(x,y)$. Finally, we shall also denote by $\mathbb{E}_{(x,y)}$ the expectation under the measure on $(\Omega,\mathcal{F})$ $\mathbb{P}_{(x,y)}(\,\cdot\,) = \mathbb{P}(\,\cdot\,|X^{\nu}_0=x,Y_0=y)$, and equivalently use the notation $\mathbb{E}[f(X^{x,y,\nu}_t,Y^y_t)] = \mathbb{E}_{(x,y)}[f(X^{\nu}_t,Y_t)]$ for any Borel-measurable function for which the expectation is well defined.

Within this setting the solution to \eqref{Xinter} and \eqref{dynY} is for any $t\geq 0$, $(x,y) \in (0,\infty) \times \mathbb{R}$ and $\nu \in \mathcal{A}(x,y)$
\begin{align}
\label{syst:XY}
\left\{
\begin{array}{lr}
\displaystyle X^{x,y,\nu}_t=e^{(\delta-g)t - \int_0^t Y^y_s ds}\Big[ x - \int_0^t e^{-(\delta-g)s + \int_0^s Y^y_u du} d\nu_s\Big], \\[+6pt]
\displaystyle Y^y_t=ye^{-\theta t} + \frac{a}{\theta}(1 - e^{-\theta t}) + \sigma e^{-\theta t} \int_0^t e^{\theta s} dW_s.
\end{array}
\right.
\end{align}

\begin{remark}
\label{rem:dimensionality}
It might be interesting to introduce a stochastic dynamics for the nominal interest rate and for the GDP growth rate. For example, one might let $\delta$ follow a diffusive mean-reverting dynamics correlated with the inflation rate, and $g$ be described by a Markov regime switching model \`a la Hamilton \cite{Hamilton}, so to capture business cycles in the GDP dynamics. This would lead to a more intricate optimisation problem with three-dimensional state space and regime switching, and the mathematical analysis of the problem will become much more challenging. We leave such interesting extension for future research.
\end{remark}

For suitable $\rho>0$, we set
\beq
\label{Markov-set}
h(\omega, t, x):= e^{-\rho t}C(x), \quad \text{and} \quad \kappa_t(\omega):= \kappa e^{-\rho t}, \qquad (\omega,t,x) \in \Omega \times \mathbb{R}_+ \times \mathbb{R},
\eeq
where $\kappa>0$, and the instantaneous cost function $C:\mathbb{R} \mapsto \mathbb{R}_+$ satisfies the following assumption.
\begin{Assumptions}
\label{ass:c}
\begin{itemize}
\item[(i)] $x \mapsto C(x)$ is strictly convex, continuously differentiable, and such that it is nondecreasing on $\mathbb{R}_+$ and $C(0)=0$; 
\item[(ii)] there exists $\gamma > 1$, $0 < K_1 < K$ and $K_2>0$ such that $$K_1|x^+|^{\gamma} - K \leq C(x) \leq K(1 + |x|^{\gamma})\quad \text{and}\quad |C'(x)| \leq K_2(1 + |x|^{\gamma-1}),$$
\end{itemize}
\end{Assumptions}
\noindent where $(\,\cdot\,)^+$ denotes the positive part of a real number. A quadratic loss function $C(x) = \frac{1}{2}x^2$ clearly satisfies Assumption \ref{ass:c}.

In the rest of this paper we make the standing assumption that the government's discount factor $\rho$ is sufficiently large.
\begin{Assumptions}
\label{ass:rho}
Let $\gamma$ as in Assumption \ref{ass:c}-\emph{(ii)} and $\rho_o > 4\Big[\delta - g - \frac{a}{\theta} + \frac{2 \sigma^2}{\theta^2}\Big] \vee 0$. Then 
$$\rho > \rho_o \vee \gamma\Big[\delta - g - \frac{a}{\theta} + \frac{\gamma \sigma^2}{2\theta^2}\Big] \vee 2(\gamma-1)\Big[\delta - g - \frac{a}{\theta} + \frac{(\gamma-1) \sigma^2}{\theta^2}\Big].$$
\end{Assumptions}
\begin{remark}
\label{rem:rho}
Assumption \ref{ass:rho} is reasonable in light of the fact that usually governments run only for a finite number of years and are therefore more concerned about present than future. Mathematically, Assumption \ref{ass:rho} takes care of the infinite time-horizon of our problem, and it ensures in particular that both the cost and the marginal cost associated to the admissible policy ``never intervene on the debt ratio" are finite.
\end{remark}

The proof of the next proposition can be found in Appendix \ref{someproofs}.
\begin{proposition}
\label{prop:verificoipotesi}
Under Assumptions \ref{ass:c} and \ref{ass:rho}, let $\kappa_t$ and $h$ be as in \eqref{Markov-set}. Then they satisfy Assumptions \ref{ass:h} and \ref{ass:kappa}, respectively.
\end{proposition}

Within such Markovian setting, we define the government's value function as 
\beq
\label{eq:valueOC}
v(x,y):=\inf_{\nu \in \mathcal{A}(x,y)}\mathcal{J}_{x,y}(\nu), \qquad (x,y) \in \mathcal{O},
\eeq
where we have set $\mathcal{O}:=(0,\infty) \times \mathbb{R}$, and where
\beq
\label{eq:J}
\mathcal{J}_{x,y}(\nu):=\mathbb{E}_{(x,y)}\bigg[\int_0^{\infty} e^{-\rho t} C(X^{\nu}_t) dt + \kappa \int_0^{\infty} e^{-\rho t} d\nu_t\bigg]
\eeq
is the total expected cost of having debt and of intervening on it.

Under Assumption \ref{ass:rho} we can prove an upper bound for the value function $v$. This is shown in the next proposition, whose proof can be found in Appendix \ref{someproofs}.
\begin{proposition}
\label{prop:valuefinite}
Set $\Theta:=\rho - \gamma\big[\delta - g - \frac{a}{\theta} + \frac{\gamma \sigma^2}{2 \theta^2}\big] >0$. Then for any $(x,y)\in \mathcal{O}$ one has
\begin{equation}
\label{eq:valueOCfinite}
\displaystyle 0 \leq v(x,y) \leq K\Big(\frac{1}{\rho} + \Theta^{-1} x^{\gamma}e^{\frac{\gamma}{\theta}|y - \frac{a}{\theta}|}\Big) \wedge \kappa x.
\end{equation}
Moreover, $v(0,y)=0$ and the mapping $x \mapsto v(x,y)$ is convex for any $y \in \mathbb{R}$.
\end{proposition}


\subsection{The Auxiliary Optimal Stopping Problem}
\label{sec:OS}

Guided by Theorem \ref{teo:ver}, in this section we introduce and solve the optimal stopping problem linked to problem \eqref{eq:valueOC}. In particular we characterise its solution in terms of an optimal stopping boundary $\hat{y}(\,\cdot\,)$ that will be shown to be the unique solution (within a certain functional class) of a nonlinear integral equation. In Section \ref{sec:OC} we will then prove that the optimal debt reduction policy is of threshold type, and it is triggered by a boundary which is closely related to $\hat{y}(\,\cdot\,)$.

Let $Y^y$ as in \eqref{syst:XY} and define
\beq
\label{def:Z}
Z^{z,y}_t := z + \int_0^t \beta_u du = z + (\delta - g) t - \int_0^t Y^y_u du, \quad (z,y) \in \mathbb{R}^2.
\eeq

Properties of the two-dimensional stochastic process $(Z^{z,y},Y^y):=\{(Z^{z,y}_t,Y^y_t), t\geq 0\}$ are collected in the following lemma. Their proof can be easily obtained from p.\ 287 of \cite{Lachal}.
\begin{lemma}
\label{lem:YZ}
The process $(Z^{z,y},Y^y):=\{(Z^{z,y}_t,Y^y_t), t\geq 0\}$ 
\begin{itemize}
\item[(i)] is strong Markov, time-homogeneous, and its infinitesimal generator is given by the second-order differential operator
\beq
\label{generator}
\mathbb{L}_{Z,Y}:= \frac{1}{2}\sigma^2 \frac{\partial^2}{\partial y^2} + (a - \theta y)\frac{\partial}{\partial y} + (\delta - g -y)\frac{\partial}{\partial z};
\eeq
\item[(ii)] has transition density 
\begin{align}
\label{transdensity}
& p_t(z,y;v,u) := \mathbb{P}\big((Z_t,Y_t) \in (dv,du)\big| Z_0=z, Y_0=y\big)/dv du \nonumber \\
& = \frac{1}{2\pi\sqrt{\Delta_t}}\exp\Big\{-\frac{\theta^2}{2\big(t - \frac{2}{\theta}\tanh(\frac{\theta t}{2})\big)}\Big[v + (\delta-g)t - z - \frac{1}{\theta}\tanh\big(\frac{\theta t}{2}(u+y-a)\big)\Big]^2 \nonumber \\
& - \frac{\theta}{1 - e^{-\theta t}}\big(u - e^{-\theta t}(y-a)\big)^2\Big\},
\end{align}
where we have set $\Delta_t:=\frac{1}{2\theta^3}(1 - e^{-2\theta t})(t - \frac{2}{\theta}\tanh(\frac{\theta t}{2}))$.
\end{itemize}
\end{lemma}

In the following we denote by $\mathbb{E}_{(z,y)}$ the expectation under the measure on $(\Omega,\mathcal{F})$ $\mathbb{P}_{(z,y)}(\,\cdot\,) = \mathbb{P}(\,\cdot\,|Z_0=z,Y_0=y)$, $(z,y) \in \mathbb{R}^2$. Moreover, in the rest of this paper we will equivalently use the notation $\mathbb{E}[f(Z^{z,y}_t,Y^y_t)] = \mathbb{E}_{(z,y)}[f(Z_t,Y_t)]$, for any Borel-measurable function for which the expectation is well defined. Before introducing the optimal stopping problem associated to problem \eqref{eq:valueOC}, we have the next technical lemma whose results will be useful in the following. Its proof is provided in Appendix \ref{someproofs}.
\begin{lemma}
\label{lem:limitZeui}
One has
\beq
\label{limitZ}
\liminf_{t \uparrow \infty} e^{-\rho t + Z_{t}} = 0, \qquad \mathbb{P}_{(z,y)}-a.s.
\eeq
and
\beq
\label{ui1}
\mathbb{E}_{(z,y)}\bigg[\int_0^{\infty} e^{-\rho s + Z_s}C'(e^{Z_s}) ds\bigg] + \kappa\mathbb{E}_{(z,y)}\bigg[\int_0^{\infty} e^{-\rho s + Z_s}|\delta - g - \rho - Y_s| ds\bigg] < \infty.
\eeq
\end{lemma}

In light of Theorem \ref{teo:ver} we introduce the optimal stopping problem with value function
\beq
\label{valueOS}
u(z,y):=\inf_{\tau \geq 0}\mathbb{E}_{(z,y)}\bigg[\int_0^{\tau} e^{-\rho t + Z_t} C'(e^{Z_t}) dt + \kappa e^{-\rho \tau + Z_{\tau}}\bigg],\quad (z,y) \in \mathbb{R}^2,
\eeq
where the optimisation is taken in the set of $\mathbb{F}$-stopping times. We denote such set by $\mathcal{T}$. In \eqref{valueOS} we also make use of the convention (cf.\ \cite{Shir}, Ch.\ 3)
\beq
\label{convention}
e^{-\rho \tau + Z_{\tau}} := \liminf_{t \uparrow \infty} e^{-\rho t + Z_{t}} = 0 \quad \text{on}\,\,\{\tau=+\infty\}.
\eeq
An integration by parts gives
\beq
\label{byparts}
e^{-\rho \tau + Z_{\tau}} = e^z + \int_{0}^{\tau} e^{-\rho t + Z_{t}}\big(\delta - g - Y_t - \rho\big) dt, \quad \mathbb{P}_{(z,y)}-a.s.,
\eeq
for any $\tau \in \mathcal{T}$, and \eqref{valueOS} may be rewritten as
\beq
\label{valueOS-2}
u(z,y):= \kappa e^z + \inf_{\tau \geq 0}\mathbb{E}_{(z,y)}\bigg[\int_0^{\tau} e^{-\rho t + Z_t} \Big(C'(e^{Z_{t}}) + \kappa(\delta - g - Y_t - \rho)\Big) dt\bigg].
\eeq
In this paper we will make use of both the equivalent representations \eqref{valueOS} and \eqref{valueOS-2}.
Notice that by \eqref{ui1} the family of random variables $\{\int_0^{\tau} e^{-\rho t + Z_t} (C'(e^{Z_{t}}) + \kappa(\delta - g - Y_t - \rho)) dt, \, \tau \in \mathcal{T}\}$ is uniformly integrable under $\mathbb{P}_{(z,y)}$, $(z,y) \in \mathbb{R}^2$. 

It is easy to see that $u(z,y) \leq \kappa e^z$ for any $(z,y) \in \mathbb{R}^2$. As usual in optimal stopping theory (see, e.g., \cite{PeskShir}), we can define the \emph{continuation region}
\beq
\label{cont}
\mathcal{C}:=\{(z,y) \in \mathbb{R}^2:\,u(z,y) < \kappa e^z\},
\eeq
and the \emph{stopping region}
\beq
\label{stopp}
\mathcal{S}:=\{(z,y) \in \mathbb{R}^2:\,u(z,y) = \kappa e^z\}.
\eeq
Because $y \mapsto Y^y$ is increasing (cf.\ \eqref{syst:XY}), the mapping $y \mapsto Z^{z,y}$ is decreasing (cf.\ \eqref{def:Z}). It thus follows from \eqref{valueOS} that $y \mapsto u(z,y)$, $z \in \mathbb{R}$, is decreasing and therefore 
\beq
\label{contstopp}
\mathcal{C}:=\{(z,y) \in \mathbb{R}^2:\,y > \hat{y}(z)\}, \qquad \mathcal{S}:=\{(z,y) \in \mathbb{R}^2:\,y \leq \hat{y}(z)\},
\eeq
for $\hat{y}$ defined as
\beq
\label{bdy}
\hat{y}(z):=\inf\{y \in \mathbb{R}:\, u(z,y) < e^z \},\quad z \in \mathbb{R},
\eeq
with the convention $\inf \emptyset = \infty$.

\begin{proposition}
\label{continuityOS}
The value function $u$ of \eqref{valueOS} (equivalently, of \eqref{valueOS-2}) is such that $(z,y)\mapsto u(z,y)$ is continuous on $\mathbb{R}^2$.
\end{proposition}
\begin{proof}

Take $(z,y) \in \mathbb{R}^2$ and let $\{(z_n,y_n), n \in \mathbb{N}\} \subset \mathbb{R}^2$ be a sequence converging to $(z,y)$. For $\varepsilon > 0$ let $\tau^{\varepsilon}:=\tau^{\varepsilon}(z,y)$ be an $\varepsilon$-optimal stopping time for $u(z,y)$. Then we can write
\begin{eqnarray}
\label{continua1}
&& u(z_n,y_n) - u(z,y) \leq \varepsilon  + \mathbb{E}\bigg[\int_0^{\tau^{\varepsilon}} e^{-\rho t + Z^{z_n,y_n}_t} C'(e^{Z^{z_n,y_n}_t}) dt - \int_0^{\tau^{\varepsilon}} e^{-\rho t + Z^{z,y}_t} C'(e^{Z^{z,y}_t}) dt\bigg] \nonumber \\
&& \hspace{1cm} + \kappa\mathbb{E}\bigg[e^{-\rho \tau^{\varepsilon}}\Big(e^{Z^{z_n,y_n}_{\tau^{\varepsilon}}} - e^{Z^{z,y}_{\tau^{\varepsilon}}}\Big)\bigg] \nonumber \\
&& = \varepsilon + \kappa(e^{z_n} - e^z) + \mathbb{E}\bigg[\int_0^{\tau^{\varepsilon}} e^{-\rho t + Z^{z_n,y_n}_t} \Big(C'(e^{Z^{z_n,y_n}_t}) - C'(e^{Z^{z,y}_t})\Big) dt\bigg] \nonumber \\
&& \hspace{1cm} + \mathbb{E}\bigg[\int_0^{\tau^{\varepsilon}} e^{-\rho t} C'(e^{Z^{z,y}_t})\Big(e^{Z^{z_n,y_n}_t} - e^{Z^{z,y}_t}\Big) dt\bigg]  \\
&& \hspace{1cm} + \kappa \mathbb{E}\bigg[\int_0^{\tau^{\varepsilon}} e^{-\rho t} (\delta - g-\rho)\Big(e^{Z^{z_n,y_n}_t} - e^{Z^{z,y}_t}\Big) dt\bigg] \nonumber \\
&& \hspace{1cm} - \kappa \mathbb{E}\bigg[\int_0^{\tau^{\varepsilon}} e^{-\rho t + Z^{z_n,y_n}_t}\Big(Y^{y_n}_t - Y^y_t\Big) dt\bigg] + \kappa \mathbb{E}\bigg[\int_0^{\tau^{\varepsilon}} e^{-\rho t} Y^y_t\Big(e^{Z^{z,y}_t} - e^{Z^{z_n,y_n}_t}\Big) dt\bigg],\nonumber
\end{eqnarray}
where the equality follows upon using \eqref{byparts} and employing simple algebra. Then, by H\"older inequality we can write from \eqref{continua1}
\begin{eqnarray}
\label{continua2}
&& u(z_n,y_n) - u(z,y) \leq \varepsilon + \kappa(e^{z_n} - e^z) \nonumber \\
&& + \mathbb{E}\bigg[\int_0^{\infty} e^{-\rho t + 2Z^{z_n,y_n}_t} dt\bigg]^{\frac{1}{2}}\mathbb{E}\bigg[\int_0^{\infty} e^{-\rho t} \big|C'(e^{Z^{z_n,y_n}_t}) - C'(e^{Z^{z,y}_t})\big|^2 dt\bigg]^{\frac{1}{2}} \nonumber \\
&& + \mathbb{E}\bigg[\int_0^{\infty} e^{-\rho t} \big|C'(e^{Z^{z,y}_t})\big|^2 dt\bigg]^{\frac{1}{2}}\mathbb{E}\bigg[\int_0^{\infty} e^{-\rho t} \big|e^{Z^{z_n,y_n}_t} - e^{Z^{z,y}_t}\big|^2 dt\bigg]^{\frac{1}{2}} \nonumber \\
&& + \kappa|\delta - g-\rho| \mathbb{E}\bigg[\int_0^{\infty} e^{-\rho t} \big|e^{Z^{z_n,y_n}_t} - e^{Z^{z,y}_t}\big| dt\bigg] \nonumber \\
&& + \kappa \mathbb{E}\bigg[\int_0^{\infty} e^{-\rho t + 2Z^{z_n,y_n}_t} dt\bigg]^{\frac{1}{2}}\mathbb{E}\bigg[\int_0^{\infty} e^{-\rho t} \big|Y^{y_n} - Y^y\big|^2 dt\bigg]^{\frac{1}{2}} \\
&& + \kappa \mathbb{E}\bigg[\int_0^{\infty} e^{-\rho t} \big|Y^y_t\big|^2 dt\bigg]^{\frac{1}{2}}\mathbb{E}\bigg[\int_0^{\infty} e^{-\rho t} \big|e^{Z^{z_n,y_n}_t} - e^{Z^{z_n,y_n}_t}\big|^2 dt\bigg]^{\frac{1}{2}}.\nonumber
\end{eqnarray}
Now, if we can apply the dominated convergence theorem, and if
\begin{itemize}
\item[(a)] $ \displaystyle \limsup_{n\uparrow \infty} \mathbb{E}\bigg[\int_0^{\infty} e^{-\rho t + 2Z^{z_n,y_n}_t} dt\bigg] < \infty$, \qquad \text{(b)} $\displaystyle \mathbb{E}\bigg[\int_0^{\infty} e^{-\rho t} \big|C'(e^{Z^{z,y}_t})\big|^2 dt\bigg] < \infty$;
\item[(c)] $\displaystyle \mathbb{E}\bigg[\int_0^{\infty} e^{-\rho t} \big|Y^y_t\big|^2 dt\bigg]<\infty$,
\end{itemize}
by taking limits as $n\uparrow \infty$ on both sides of \eqref{continua2} we can conclude that
\beq
\label{continua3}
\limsup_{n\uparrow \infty}u(z_n,y_n) \leq u(z,y) + \varepsilon.
\eeq

On the other hand, let $\tau^{\varepsilon}_n:=\tau^{\varepsilon}(z_n,y_n)$, $\varepsilon > 0$, be an $\varepsilon$-optimal stopping time for $u(z_n,y_n)$, and evaluate $u(z,y) - u(z_n,y_n)$ by employing estimates analogous to those used in \eqref{continua1}-\eqref{continua2} so to find
\beq
\label{continua4}
\liminf_{n\uparrow \infty}u(z_n,y_n) \geq u(z,y) - \varepsilon,
\eeq
whenever the dominated convergence theorem applies and (a)-(c) hold true. By letting $\varepsilon \downarrow 0$ in \eqref{continua3} and \eqref{continua4}, we find that $u$ is continuous at $(z,y)$, and we therefore conclude since $(z,y)$ was arbitrary in $\mathbb{R}^2$.

To complete the proof it thus remains to prove (a)-(c) above, and that we can actually use the dominated convergence theorem.
Using Tonelli's theorem, the results collected in Lemma \ref{lemm:Y}, and Assumption \ref{ass:rho} we find for any $(z_o,y_o)\in \mathbb{R}^2$
\beq
\label{stimaa}
\mathbb{E}\bigg[\int_0^{\infty} e^{-\rho t + 2Z^{z_o,y_o}_t} dt\bigg] \leq K\,e^{2z_o + \frac{2}{\theta}|y_o-\frac{a}{\theta}|},
\eeq
for some $K>0$. Equation \eqref{stimaa} in turn implies (a).
On the other hand, by using the assumption on the growth of $C'(\,\cdot\,)$ (cf.\ Assumption \ref{ass:c}-\emph{(ii)}), and repeating the arguments leading to \eqref{stimaa} above, we have for any $(z_o,y_o)\in \mathbb{R}^2$
\beq
\label{stimab}
\mathbb{E}\bigg[\int_0^{\infty} e^{-\rho t} \big|C'(e^{Z^{z_o,y_o}_t})\big|^2 dt\bigg] \leq K\Big(1 + e^{2(\gamma -1)(z_o + \frac{1}{\theta}|y_o - \frac{a}{\theta}|)}\Big),
\eeq
for a suitable $K>0$. Then (b) follows.
Finally, (c) holds because for any $y_o \in \mathbb{R}$ one has $\mathbb{E}[\int_0^{\infty} e^{-\rho t} |Y^{y_o}_t|^2 dt] \leq K(1 + |y_o|^2)$ for some $K>0$.

To verify the assumptions of the dominated convergence theorem it suffices to consider $\{(z_n,y_n), n \in \mathbb{N}\} \subset (z-\eta, z + \eta) \times (y-\eta, y + \eta)$, for a suitable $\eta>0$. By recalling the convexity of $C(\,\cdot\,)$ we then have a.s.\ for any $t\geq 0$
$$\big|C'(e^{Z^{z_n,y_n}_t}) - C'(e^{Z^{z,y}_t})\big|^2 \leq 2\big[|C'(e^{Z^{z+\eta,y-\eta}_t})|^2 +  |C'(e^{Z^{z,y}_t})|^2\big], \qquad \big|Y^{y_n}_t - Y^y_t\big|^2 \leq |y_n-y|^2 \leq \eta^2,$$
and
$$\big|e^{Z^{z_n,y_n}_t} - e^{Z^{z,y}_t}\big|^2 \leq 2\big[e^{2Z^{z+\eta,y-\eta}_t} + e^{2Z^{z,y}_t}\big].$$
All the quantities on the right-hand sides of the three equations above are integrable with respect to the product measure $\mathbb{P} \otimes e^{-\rho t} dt$ thanks to \eqref{stimaa} and \eqref{stimab}. It thus follows that the dominated convergence theorem can be applied, and this completes the proof.
\end{proof}

From Proposition \ref{continuityOS} it follows that the stopping set $\mathcal{S}$ of \eqref{stopp} is closed, and the continuation region $\mathcal{C}$ of \eqref{cont} is open. Recalling \eqref{contstopp}, thanks to Lemma \ref{lem:limitZeui} we have by Theorem D.12 in Appendix D of \cite{KS-MF} that the stopping time
\begin{equation}
\label{OStime}
\tau^{\star}(z,y):=\inf\{t\geq 0:\,(Z^{z,y}_t,Y^y_t) \in \mathcal{S}\} = \inf\{t\geq 0:\,Y^y_t \leq \hat{y}(Z^{z,y}_t)\}
\end{equation}
is optimal for problem \eqref{valueOS}. In the next proposition we rule out the possibility that the stopping set $\mathcal{S}$ is empty. 
\begin{proposition}
\label{stoppingnoempty}
The stopping region $\mathcal{S}$ of \eqref{stopp} is not empty.
\end{proposition}

\begin{proof}
Assume by contradiction that $\mathcal{S}=\{(z,y)\in \mathbb{R}^2:\, u(z,y) = \kappa e^z\}=\emptyset$. This would imply that for any $(z,y)\in \mathbb{R}^2$ one has (cf.\ \eqref{convention})
\begin{align}
\label{contra1}
&& \kappa e^z > u(z,y) = \mathbb{E}_{(z,y)}\bigg[\int_0^{\infty} e^{-\rho t + Z_t} C'(e^{Z_t}) dt\bigg] \geq \mathbb{E}_{(z,y)}\bigg[\int_0^{\infty} e^{-\rho t} C(e^{Z_t}) dt\bigg] \nonumber \\
&& \geq K_1 \mathbb{E}_{(z,y)}\bigg[\int_0^{\infty} e^{-\rho t + \gamma Z_t} dt\bigg] - \frac{K}{\rho} = K_1 e^{\gamma z}\mathbb{E}\bigg[\int_0^{\infty} e^{-\rho t + \gamma Z^{0,y}_t} dt\bigg]  - \frac{K}{\rho},
\end{align}
where we have used the convexity of $x \mapsto C(x)$, the fact that $C(0)=0$, and Assumption \ref{ass:c}-\emph{(ii)}.
It thus follows from \eqref{contra1} that
\begin{equation}
\label{contra2}
\kappa > e^{(\gamma - 1) z} K_1 \mathbb{E}\bigg[\int_0^{\infty} e^{-\rho t + \gamma Z^{0,y}_t} dt\bigg] - \frac{K e^{-z}}{\rho}.
\end{equation}
Since the expected value on the right-hand side of \eqref{contra2} is finite by Assumption \ref{ass:rho}, recalling that $\gamma > 1$ we therefore reach a contradiction in \eqref{contra2} by taking $z$ sufficiently large, for any given and fixed $y\in \mathbb{R}$.
\end{proof}

Some properties of $\hat{y}$ are collected in the following proposition.

\begin{proposition}
\label{prop:preliminarybdy}
Let $\hat{y}$ be defined as in \eqref{bdy}. Then the following properties hold true:
\begin{itemize}
\item[(i)] $\hat{y}(z) \leq \frac{C'(e^z)}{\kappa} + \delta - g - \rho$ for any $z\in \mathbb{R}$;
\item[(ii)] $z \mapsto \hat{y}(z)$ is nondecreasing;
\item[(iii)] $z \mapsto \hat{y}(z)$ is right-continuous.
\end{itemize}
\end{proposition}
\begin{proof}

We prove each item of the proposition separately.
\vspace{0.15cm}

\emph{(i)}\,\, From \eqref{valueOS-2} one can easily see that it is never optimal to stop the evolution of $(Z,Y)$ in the region $\mathcal{U}:=\{(z,y)\in \mathbb{R}^2:\, C'(e^z) + \kappa(\delta - g - y - \rho) < 0\}$. That is $\mathcal{U} \subset \mathcal{C}$. Hence $\mathcal{S} \subseteq \mathcal{U}^c$ and $\hat{y}(z) \leq \frac{C'(e^z)}{\kappa} + \delta - g - \rho$ for any $z\in \mathbb{R}$.
\vspace{0.08cm}

\emph{(ii)}\,\, Let $(z_1,y) \in \mathcal{S}$ be given and fixed, take an arbitrary $z_2 \geq z_1$, and let $\tau^{\varepsilon}:=\tau^{\varepsilon}(z_2,y)$ be $\varepsilon$-optimal for $u(z_2,y)$. Then we have from \eqref{valueOS-2}
\begin{align}
\label{bd-1}
& 0 \geq u(z_2,y) - \kappa e^{z_2} \geq \mathbb{E}\bigg[\int_0^{\tau^{\varepsilon}} e^{-\rho t + z_2 + Z^{0,y}_t}\Big(C'(e^{z_2 + Z^{0,y}_t}) + \kappa(\delta - g - \rho -Y^y_t)\Big)dt\bigg] - \varepsilon \nonumber \\
& = \mathbb{E}\bigg[\int_0^{\tau^{\varepsilon}} e^{-\rho t + z_1 + Z^{0,y}_t} e^{z_2-z_1}\Big(C'(e^{z_1 + Z^{0,y}_t}e^{z_2-z_1}) + \kappa(\delta - g - \rho - Y^y_t)\Big)dt\bigg] - \varepsilon \nonumber \\
& \geq e^{z_2-z_1} \mathbb{E}\bigg[\int_0^{\tau^{\varepsilon}} e^{-\rho t + Z^{z_1,y}_t}\Big(C'(e^{ Z^{z_1,y}_t}) + \kappa(\delta - g - \rho - Y^y_t)\Big)dt\bigg] - \varepsilon \nonumber \\
& \geq e^{z_2-z_1} \big(u(z_1,y) - \kappa e^{z_1}\big) - \varepsilon = - \varepsilon,\nonumber
\end{align}
where we have used that $\tau^{\varepsilon}$ is suboptimal for $u(z_1,y)$ and that $x \mapsto C'(x)$ is nondecreasing by convexity. It thus follows by arbitraryness of $\varepsilon > 0$ that $(z_2,y) \in \mathcal{S}$ for any $z_2 \geq z_1$, and therefore that $z \mapsto \hat{y}(z)$ is nondecreasing.
\vspace{0.08cm}

\emph{(iii)}\,\,By continuity of $u$ it follows that $\hat{y}(\,\cdot\,)$ is upper semi-continuous, hence it is right-continuous since it is nondecreasing by (ii).
\end{proof}

We now continue by improving the regularity of the value function \eqref{valueOS}. Namely, we now show that the well known \emph{smooth-fit principle} holds, by proving that $u \in C^1(\mathbb{R}^2)$. The proof relies of an application of an interesting result obtained by S.D.\ Jacka in \cite{Jacka} (cf.\ Corollary 7 in Section 4 of \cite{Jacka}).

\begin{proposition}
\label{prop:valueOSC1}
The value function $u$ of \eqref{valueOS} (equivalently, of \eqref{valueOS-2}) is such that $u \in C^1(\mathbb{R}^2)$.
\end{proposition}

\begin{proof}
First of all we notice that an application of strong Markov property allows to write
\beq
\label{newOS}
u(z,y) = \kappa e^z + g(z,y) - f(z,y),
\eeq
where we have set
\beq
\label{g}
g(z,y):=\mathbb{E}_{(z,y)}\bigg[\int_0^{\infty}e^{-\rho t + Z_t}\Big(C'(e^{Z_t}) + \kappa(\delta - g - \rho - Y_t)\Big)dt\bigg],
\eeq
and
\beq
\label{f}
f(z,y):=\sup_{\tau \geq 0}\mathbb{E}_{(z,y)}\Big[e^{-\rho\tau}g(Z_{\tau},Y_{\tau})\Big].
\eeq
Hence, the $C^1$ property of $u$ reduces to check that for $g$ and $f$.

By \eqref{transdensity} we can write
\begin{align*}
& g(z,y)= \int_0^{\infty} e^{-\rho t } \bigg(\int_{\mathbb{R}^2} e^v\big(C'(e^{v}) + \kappa(\delta - g - \rho - u)\big)\, p_t(z,y;v,u)\,dv du\bigg) dt
\end{align*}
and an application of dominated convergence theorem shows that $g \in C^1(\mathbb{R}^2)$.

It thus remain to check for the $C^1$ property of $f$. With regard to the notation of \cite{Jacka} we set $\xi_t:=(Z_t,Y_t)$,
$$X_t:=e^{-\rho t}g(\xi_t) = \mathbb{E}\bigg[\int_t^{\infty}e^{-\rho s + Z_s}\Big(C'(e^{Z_s}) + \kappa(\delta - g - \rho - Y_s)\Big)ds\Big| \mathcal{F}_t\bigg],$$
and we can write $X_t = M_t + A_t$ where,
$$M_t:=\mathbb{E}\bigg[\int_0^{\infty}e^{-\rho s + Z_s}\Big(C'(e^{Z_s}) + \kappa(\delta - g - \rho - Y_s)\Big)ds\Big| \mathcal{F}_t\bigg]$$
and
$$A_t:=-\int_0^{t}e^{-\rho s + Z_s}\Big(C'(e^{Z_s}) + \kappa(\delta - g - \rho - Y_s)\Big)ds = \int_0^t \big(dA^{+}_s + dA^{-}_s\big).$$
Notice that $M$ is a uniformly integrable martingale thanks to \eqref{ui1} of Lemma \ref{lem:limitZeui}, and $dA^+$ and $dA^{-}$ above are given by
$$dA^{\pm}_s:= e^{-\rho s + Z_s}\Big(C'(e^{Z_s}) + \kappa(\delta - g - \rho - Y_s)\Big)^{\mp}ds,$$
which are clearly absolutely continuous with respect to Lebesgue measure $dm_2:=dt$. Moreover, the set $\partial \mathcal{D}$ in \cite{Jacka} reads in our case as $\{(z,y) \in \mathbb{R}^2:\,y=\hat{y}(z,y)\}$, which has zero measure with respect to $dm_1:=dzdy$. Finally, the process $\xi:=(Z,Y)$ has density with respect to $m_1$ which has spatial derivatives uniformly continuous in $\mathbb{R}^2 \times [t_0,t_1]$, for any $0<t_0 < t_1 < \infty$ (see \eqref{transdensity}). Hence, Corollary $7$ in \cite{Jacka} holds and the proof is complete.
\end{proof}

We now exploit the fact that the process $Z$ is of bounded variation to obtain additional regularity for $u$. A similar idea has been recently employed in a different context also in \cite{PeskirJohnson}, Corollary 14.

\begin{proposition}
\label{cor:C2}
One has that $u \in C^{1,2}(\overline{\mathcal{C}})$, where $\overline{\mathcal{C}}:=\{(z,y)\in \mathbb{R}^2:\, y \geq \hat{y}(z)\}$.
\end{proposition}
\begin{proof}
The proof is organized in two steps. 
\vspace{0.25cm}

\noindent \emph{Step 1.} Here we show that $u\in C^{1,2}$ inside $\Delta:= \{(z,y) \in \mathcal{C}:\, \delta - g - y \neq 0\}$. We accomplish that by showing that $u\in C^{1,2}$ inside the two sets $\Delta_-:=\{(z,y) \in \mathcal{C}:\, \delta - g - y<0\}$ and $\Delta_+:=\{(z,y) \in \mathcal{C}:\, \delta - g - y>0\}$, that are clearly such that $\Delta = \Delta_+ \cup \Delta_-$. 

Let $(z_o,y_o) \in \Delta_-$ be given and fixed. Then take $\varepsilon > 0$, $z_1 < z_o < z_2$ and define the rectangle $\mathcal{R}_-:=(z_1,z_2) \times (y_o-\varepsilon, y_o + \varepsilon)$ such that its closure $\cl(\mathcal{R}_-) \subset \Delta_-$. Denote by $\partial_o\mathcal{R}_-:= \partial \mathcal{R}_- \setminus [\{z_2\}\times (y_o-\varepsilon, y_o + \varepsilon)]$ the parabolic boundary of this rectangle, and consider the Dirichlet problem
\begin{align}
\label{DirichletP}
\left\{
\begin{array}{ll}
\displaystyle - f_z + \Big(\frac{\sigma^2}{2|\delta - g - y|}\Big) f_{yy} + \Big(\frac{a-\theta y}{|\delta - g - y|}\Big)f_y - \Big(\frac{\rho}{|\delta - g - y|}\Big) f = -\frac{e^z C'(e^z)}{|\delta - g - y|}, & \text{on}\,\,\mathcal{R}_-, \\[+6pt]
f = u, & \text{on}\,\,\partial_o\mathcal{R}_-.
\end{array}
\right.
\end{align}
Notice that the first equation in \eqref{DirichletP} is equivalent in $\mathcal{R}_-$ to the more familiar partial differential equation $(\mathbb{L}_{Z,Y} - \rho\big)u(z,y) = -e^{z}C'(e^z)$. 

Since all the coefficients in the first equation of \eqref{DirichletP} are smooth and bounded in $\mathcal{R}_-$, the volatility coefficient is uniformly elliptic in $\mathcal{R}_-$ as $|\delta - g - y|^{-1} \geq |\delta - g - y_o-\varepsilon|^{-1}>0$, and $u$ is continuous, by classical theory on parabolic PDEs (see, e.g., \ Chapter V of \cite{Liebermann}) problem \eqref{DirichletP} admits a unique solution with $f_z,f_y,f_{yy}$ continuous. 

\vspace{0.1cm}

We now show that $f$ coincides with $u$ in $\mathcal{R}_{-}$. For $(z,y) \in \mathcal{R}_-$, define the stopping time $\vartheta:=\inf\{t \geq 0:\, (Z,Y) \in \partial_o\mathcal{R}_-\}$, $\mathbb{P}_{(z,y)}$-a.s., and set $\vartheta_n:=\vartheta \wedge n$, $n \in \mathbb{N}$. Then by Dynkin's formula
\begin{eqnarray}
\label{Dynkin-cor}
&& f(z,y) = \mathbb{E}_{(z,y)}\bigg[e^{-\rho \vartheta_n} f(Z_{\vartheta_n}, Y_{\vartheta_n}) + \int_0^{\vartheta_n} e^{-\rho s + Z_s} C'(e^{Z_s}) ds\bigg]. \nonumber 
\end{eqnarray}
Letting $n\uparrow \infty$ on both sides of the previous equation, and noticing that $|f(Z^{z,y}_{\vartheta_n}, Y^y_{\vartheta_n})| \leq C$, for some constant $C>0$ independent of $n$, we obtain
\begin{equation}
\label{Dynkin-cor2}
f(z,y) = \mathbb{E}_{(z,y)}\bigg[e^{-\rho \vartheta} u(Z_{\vartheta}, Y_{\vartheta}) + \int_0^{\vartheta} e^{-\rho s + Z_s} C'(e^{Z_s}) ds\bigg] = u(z,y),
\end{equation}
where the last equality follows from the subharmonic property of $u$ (cf.\ \cite{PeskShir}, Ch.\ I, Sec.\ 2, Th.\ 2.4), upon recalling that $\vartheta \leq \tau^{\star}$ $\mathbb{P}_{(z,y)}$-a.s.\ since $\mathcal{R}_{-} \subset \mathcal{C}$. Hence $f = u$ on $\mathcal{R}_-$. By arbitrariness of $\mathcal{R}_{-}$ we conclude that $u \in C^{1,2}(\Delta_-)$, and there it uniquely solves the first equation in \eqref{DirichletP}; that is, $(\mathbb{L}_{Z,Y} - \rho\big)u(z,y) = -e^{z}C'(e^z)$ on $\Delta_-$. 

By analogous arguments one can show that $u \in C^{1,2}(\Delta_+)$, and therefore conclude that $u \in C^{1,2}(\Delta)$ and solves
\begin{equation}
\label{eq:C12}
\frac{1}{2}\sigma^2 u_{yy}(z,y) = \rho u(z,y) - (a - \theta y)u_y(z,y)- (\delta - g- y)u_z(z,y) - e^{z}C'(e^z), \qquad (z,y) \in \Delta.
\end{equation}
\vspace{0.25cm}

\noindent \emph{Step 2.} Notice that the right-hand side of \eqref{eq:C12} involves only $u$, its first derivatives and other continuous functions, and it is therefore continuous on $\overline{\mathcal{C}}$ by Proposition \ref{prop:valueOSC1}. Since $\Delta$ is dense in $\overline{\mathcal{C}}$, $u_{yy}$ admits a continuous extension to $\overline{\mathcal{C}}$, that we denote by $\overline{u}_{yy}$.

Then taking an arbitrary $(z_o,y_o) \in \mathcal{C}$ we can write
$$u_y(z_o,y) = u_y(z_o,y_o) + \int_{y_o}^{y} \overline{u}_{yy}(z_o,\xi)\,d\xi, \quad \forall y \geq \hat{y}(z_o),$$
and the latter yields $u_y \in C^{0,1}(\overline{\mathcal{C}})$. Since we already know that $u \in C^1(\mathbb{R}^2)$ (cf.\ Proposition \ref{prop:valueOSC1}), we therefore conclude that $u \in C^{1,2}(\overline{\mathcal{C}})$.

\end{proof}

From the results collected above it follows that $u$ solves the free-boundary problem
\begin{align}
\label{FBP}
\left\{
\begin{array}{ll}
\big(\mathbb{L}_{Z,Y} - \rho\big)u(z,y) = -e^{z}C'(e^z), & y > \hat{y}(z),\,\,z\in\mathbb{R}\\[+6pt]
u(z,y) = \kappa e^{z}, & y \leq \hat{y}(z),\,\,z\in\mathbb{R}\\[+6pt]
u_z(z,y) = \kappa e^{z}, & y = \hat{y}(z),\,\,z\in\mathbb{R}\\[+6pt]
u_y(z,y) = 0, & y = \hat{y}(z),\,\,z\in\mathbb{R},
\end{array}
\right.
\end{align}
with $u \in C^{1,2}$ inside $\overline{\mathcal{C}}$. We now show that the boundary $\hat{y}$ is in fact a continuous function.

\begin{proposition}
\label{prop:bdcont}
The optimal stopping boundary $\hat{y}$ is such that $z \mapsto \hat{y}(z)$ is continuous. 
\end{proposition}
\begin{proof}

Thanks to Proposition \ref{prop:preliminarybdy} it suffices to prove that $\hat{y}(\,\cdot\,)$ is left-continuous. We do that by employing a contradiction scheme inspired by that in \cite{DeA2015}. This is possible since the process $Z$ is of bounded variation and therefore it behaves as a ``time-like" variable. Assume that there exists some $z_o\in \mathbb{R}$ such that $\hat{y}(z_o-) < \hat{y}(z_o)$, where we have set $\hat{y}(z_o-):=\lim_{\varepsilon \downarrow 0}\hat{y}(z_o-\varepsilon)$. Such limit exists by monotonicity of $\hat{y}(\,\cdot\,)$. Then we can choose $y_1,y_2$ such that $\hat{y}(z_o-)<y_1<y_2<\hat{y}(z_o)$, $z_1 < z_o$ and define a rectangular domain with vertices $(z_o,y_1)$, $(z_o,y_2)$, $(z_1,y_1)$, $(z_2,y_2)$. 

Noticing that $(z_1,z_o) \times (y_1,y_2) \subset \mathcal{C}$ and $\{z_o\} \times [y_1,y_2] \subset \mathcal{S}$, from \eqref{FBP} $u$ solves
\begin{align}
\label{CauchyP}
\left\{
\begin{array}{ll}
\big(\mathbb{L}_{Z,Y} - \rho\big)u(z,y) = -e^{z}C'(e^z), & (z,y)\in (z_1,z_o) \times (y_1,y_2), \\[+6pt]
u(z_o,y) = \kappa e^{z_o}, & y \in [y_1,y_2], 
\end{array}
\right.
\end{align}

Denote by $C^{\infty}_c(y_1,y_2)$ the set of functions with infinitely many continuous derivatives and compact support in $(y_1,y_2)$. Pick an arbitrary $\psi \geq 0$ from $C^{\infty}_c(y_1,y_2)$ such that $\int_{y_1}^{y_2} \psi(y)dy>0$, multiply both sides of the first equation in \eqref{CauchyP} by $\psi$, and integrate over $(y_1,y_2)$ so to obtain
\begin{equation}
\label{latter}
-\int_{y_1}^{y_2} e^{z}C'(e^z) \psi(y) dy = \int_{y_1}^{y_2} \big(\mathbb{L}_{Z,Y} - \rho\big)u(z,y)\psi(y) dy, \quad z\in [z_1,z_o).
\end{equation}
Then recalling \eqref{generator}, integrating by parts with respect to $y$ the terms in the right-hand side of \eqref{latter} which involve the first and the second derivatives in the $y$-direction, one finds
\beq
\label{conbd1}
-\int_{y_1}^{y_2} e^{z}C'(e^z) \psi(y) dy = \int_{y_1}^{y_2} \big(\mathbb{L}^*_{Y}\psi\big)(y)u(z,y)dy + \int_{y_1}^{y_2} \big(\delta - g - y\big)\psi(y)u_z(z,y)dy, 
\eeq
for $z\in [z_1,z_o)$, and where $\mathbb{L}^*_{Y}$ is the second-order differential operator which acting on a function $f \in C^2(\mathbb{R})$ yields
$$\mathbb{L}^*_{Y}f:=\frac{1}{2}\sigma^2 \frac{\partial^2 f}{\partial y^2}- \frac{\partial}{\partial y}\big((a - \theta y)f\big) - \rho f.$$

Taking limits as $z \uparrow z_o$ on both sides of \eqref{conbd1} above, invoking dominated convergence theorem and recalling continuity of $u_z$ on $\mathbb{R}^2$ (cf.\ Proposition \ref{prop:valueOSC1}) one has
\beq
\label{conbd2}
-\int_{y_1}^{y_2} e^{z_o}C'(e^{z_o}) \psi(y) dy = \int_{y_1}^{y_2} \big(\mathbb{L}^*_{Y}\psi\big)(y)u(z_o,y)dy + \int_{y_1}^{y_2} \big(\delta - g - y\big)\psi(y)u_z(z_o,y)dy.
\eeq
Since $u_z(z_o,y) = \kappa e^{z_o} = u(z_o,y)$ for any $y \in [y_1,y_2]$, then rearranging terms in \eqref{conbd2} gives
\beq
\label{conbd3}
-\int_{y_1}^{y_2} e^{z_o}\big(\kappa(\delta - g - \rho - y) + C'(e^{z_o})\big)\psi(y)dy = \int_{y_1}^{y_2} \kappa e^{z_o}\Big[\frac{1}{2}\sigma^2 \frac{\partial^2 \psi}{\partial y^2}(y) - \frac{\partial}{\partial y}\big((a - \theta y)\psi\big)(y)\Big] dy.\nonumber
\eeq
Because $y_2 < \hat{y}(z_o) \leq \delta - g - \rho + \frac{C'(e^{z_o})}{\kappa}$ (cf.\ Proposition \ref{prop:preliminarybdy}) and $\psi \geq 0$, the left-hand side of the last equation is strictly negative. On the other hand, an integration reveals that the right-hand side of the latter equals zero because $\psi \in C^{\infty}_c(y_1,y_2)$. Hence we reach a contradiction and the proof is complete.
\end{proof}

The next result provides the integral representation of the value function $u$ of problem \eqref{valueOS}. Such representation will then allow us to obtain an integral equation for the stopping boundary $\hat{y}$ (cf.\ Theorem \ref{th:integral} below).
\begin{proposition}
\label{prop:representingu}
Let $\hat{y}(\,\cdot\,)$ be the stopping boundary of \eqref{bdy}. Then for any $(z,y) \in \mathbb{R}^2$ the value function $u$ of \eqref{valueOS} can be written as
\begin{align}
\label{OSrepresent}
&u(z,y)=\mathbb{E}_{(z,y)}\bigg[\int_{0}^{\infty}e^{-\rho s + Z_s}C'(e^{Z_s})\mathds{1}_{\{Y_s > \hat{y}(Z_s)\}} ds\bigg] \\
& \hspace{1.7cm}-\mathbb{E}_{(z,y)}\bigg[\int_{0}^{\infty}e^{-\rho s + Z_s}\kappa \big(\delta-g-\rho-Y_s\big)\mathds{1}_{\{Y_s \leq \hat{y}(Z_s)\}} ds\bigg].\nonumber
\end{align}
\end{proposition}
\begin{proof}
The proof is based on an application of a generalised version of It\^o's lemma. Let $R>0$ and define $\tau_R:=\inf\{t\geq 0:\, |Y_t| \geq R\,\,\text{or}\,\,|Z_t| \geq R\}$ under $\mathbb{P}_{(z,y)}$. Since $u \in C^1(\mathbb{R}^2)$ and $u_{yy} \in L^{\infty}_{loc}(\mathbb{R}^2)$ (cf.\ Propositions \ref{prop:valueOSC1} and \ref{cor:C2}), we can apply a weak version of It\^o's lemma (see, e.g., \cite{BensoussanLions}, Lemma $8.1$ and Theorem $8.5$, pp.\ 183--186) up to the stopping time $\tau_R \wedge T$, for some $T>0$, so to obtain
\begin{align}
\label{repru-1}
&u(z,y)=\mathbb{E}_{(z,y)}\bigg[e^{-\rho(\tau_R \wedge T)}u(Z_{\tau_R \wedge T},Y_{\tau_R \wedge T}) - \int_{0}^{\tau_R \wedge T}e^{-\rho s} \big(\mathbb{L}_{Z,Y} - \rho\big)u(Z_s,Y_s) ds\bigg].
\end{align}
The right-hand side of \eqref{repru-1} is well defined, since the transition probability of $(Z,Y)$ is absolutely continuous with respect to the Lebesgue measure (cf. Lemma \ref{lem:YZ}) and $(\mathbb{L}_{Z,Y} - \rho)u$ is defined up to a set of zero Lebesgue measure.

Since $u$ solves the free-boundary problem \eqref{FBP} we have
$$\big(\mathbb{L}_{Z,Y} - \rho\big)u(z,y)= - e^{z}C'(e^{z})\mathds{1}_{\{y > \hat{y}(z)\}} + \kappa(\delta-g-\rho-y)e^{z}\mathds{1}_{\{y \leq \hat{y}(z)\}} \qquad \text{for \,\,a.a.}\,\,(z,y)\in \mathbb{R}^2,$$
and using again that the transition probability of $(Z,Y)$ is absolutely continuous with respect to the Lebesgue measure (cf.\ \eqref{transdensity}) equation \eqref{repru-1} rewrites as
\begin{align}
\label{repru-2}
&u(z,y)=\mathbb{E}_{(z,y)}\bigg[e^{-\rho(\tau_R \wedge T)}u(Z_{\tau_R \wedge T},Y_{\tau_R \wedge T}) + \int_{0}^{\tau_R\wedge T}e^{-\rho s + Z_s}C'(e^{Z_s})\mathds{1}_{\{Y_s > \hat{y}(Z_s)\}} ds\bigg] \\
& \hspace{1.7cm}-\mathbb{E}_{(z,y)}\bigg[\int_{0}^{\tau_R \wedge T}e^{-\rho s + Z_s}\kappa \big(\delta-g-\rho-Y_s\big)\mathds{1}_{\{Y_s \leq \hat{y}(Z_s)\}} ds\bigg].\nonumber
\end{align}
By taking limits as $R\uparrow \infty$ and $T\uparrow \infty$ in \eqref{repru-2}, and employing Lemma \ref{lemm:limits} of Appendix \ref{auxiliaryres}, we conclude that \eqref{OSrepresent} holds true.
\end{proof}

Setting
\begin{equation}
\label{def:H}
H_{\hat{y}}(z,y):=e^{z}C'(e^z)\mathds{1}_{\{y > \hat{y}(z)\}} - \kappa(\delta - g - \rho - y)e^z \mathds{1}_{\{y \leq \hat{y}(z)\}}, \quad (z,y) \in \mathbb{R}^2,
\end{equation}
it follows that \eqref{OSrepresent} takes the form
\beq
\label{OSrepresent-new}
u(z,y)=\mathbb{E}_{(z,y)}\bigg[\int_0^{\infty} e^{-\rho s} H_{\hat{y}}(Z_s,Y_s) ds\bigg].
\eeq
Since $|H_{\hat{y}}(z,y)| \leq e^{z}C'(e^z) + \kappa|\delta - g - \rho - y|e^z$, $(z,y)\in \mathbb{R}^2$, it is a consequence of \eqref{ui1} in Lemma \ref{lem:limitZeui} that  $H_{\hat{y}}(Z_s,Y_s) \in L^1(\mathbb{P}_{(z,y)} \otimes e^{-\rho s}ds)$. This fact, together with the strong Markov property and standard arguments based on conditional expectations applied to representation formula \eqref{OSrepresent} (equivalently \eqref{OSrepresent-new}) allow to conclude that for any stopping time $\tau \in \mathcal{T}$ and $(z,y) \in \mathbb{R}^2$
\begin{align}
\label{subharmonic0}
& e^{-\rho \tau}u(Z^{z,y}_{\tau},Y^y_{\tau}) + \int_0^{\tau} e^{-\rho s}H_{\hat{y}}(Z^{z,y}_s,Y^y_s)ds = \mathbb{E}_{(z,y)}\bigg[\int_0^{\infty} e^{-\rho s}H_{\hat{y}}(Z_s,Y_s)ds\Big|\mathcal{F}_{\tau}\bigg].
\end{align}
In particular,
\begin{align}
\label{subharmonic1}
&\Big\{e^{-\rho t}u(Z^{z,y}_{t},Y^y_{t}) + \int_0^{t} e^{-\rho s}H_{\hat{y}}(Z^{z,y}_s,Y^y_s)ds, \,t\geq 0\Big\}\,\,\text{is a uniformly integrable $\mathcal{F}_t$-martingale},
\end{align}
and for any stopping time $\tau \in \mathcal{T}$
\begin{align}
\label{subharmonic3}
& e^{-\rho \tau}u(Z^{z,y}_{\tau},Y^y_{\tau}) \leq \mathbb{E}_{(z,y)}\bigg[\int_0^{\infty} e^{-\rho s}\big|H_{\hat{y}}(Z_s,Y_s)\big|ds\Big|\mathcal{F}_{\tau}\bigg], \quad (z,y) \in \mathbb{R}^2.
\end{align}
Hence the family of random variables $\{e^{-\rho \tau}u(Z^{z,y}_{\tau},Y^y_{\tau}),\,\tau \in \mathcal{T}\}$ is uniformly integrable.

We now continue our analysis by providing the integral characterisation of $\hat{y}(\,\cdot\,)$.
\begin{theorem}
\label{th:integral}
Recall \eqref{transdensity} and let
\begin{eqnarray*}
\mathcal{M}& \hspace{-0.25cm}:= \hspace{-0.25cm} & \Big\{f:\mathbb{R} \mapsto \mathbb{R}\,\,\text{continuous, nondecreasing, dominated from above by} \nonumber \\ 
&& \hspace{4cm} \vartheta(z):=\frac{C'(e^z)}{\kappa} + \delta - g - \rho\Big\}. 
\end{eqnarray*}

Then the boundary $\hat{y}(\,\cdot\,)$ is the unique function in $\mathcal{M}$ solving the nonlinear integral equation
\begin{align}
\label{inteq-2}
& \kappa e^z = \int_0^{\infty}e^{-\rho t}\Big(\int_{\mathbb{R}^2} e^{v}C'(e^v) \mathds{1}_{\{u > f(v)\}} p_t(z,f(z);v,u)du dv\Big) dt \nonumber \\
& - \int_0^{\infty}e^{-\rho t}\Big(\int_{\mathbb{R}^2} \kappa e^{v}(\delta - g - \rho - u) \mathds{1}_{\{u \leq f(v)\}} p_t(z,f(z);v,u)du dv\Big) dt.
\end{align}
\end{theorem}
\begin{proof}
The proof is organised in several steps. It follows arguments similar to those recently employed in \cite{DeAFeFe} for a two-dimensional elliptic problem, and it extends to our two-dimensional setting the arguments presented in \cite{PeskShir}, Section 25.\vspace{0.15cm}

\emph{Step 1.} To show existence of a solution to \eqref{inteq-2} it suffices to show that $\hat{y}$ of \eqref{bdy} solves it. To see this we notice that $\hat{y} \in \mathcal{M}$ by Propositions \ref{prop:preliminarybdy} and \ref{prop:bdcont}. Moreover, evaluating both sides of \eqref{OSrepresent} at $y = \hat{y}(z)$, $z \in \mathbb{R}$, one finds \eqref{inteq-2}, upon using that $u(z,\hat{y}(z))=\kappa e^z$ and expressing the expected value as an integral with respect to the probability density function \eqref{transdensity} of the process $(Z,Y)$.
\vspace{0.25cm}

We now move to proving uniqueness and we argue by a contradiction scheme by supposing that there exists another function $b \in \mathcal{M}$ solving \eqref{inteq-2}. In the following steps we will show that actually $\hat{y}(z) \leq b(z)$ for any $z \in \mathbb{R}$ (cf.\ \emph{Step 2}), and also that $\hat{y}(z) \geq b(z)$ for any $z \in \mathbb{R}$ (cf.\ \emph{Step 3}). To accomplish that for $(z,y) \in \mathbb{R}^2$ we define 
\begin{equation}
\label{def:w}
w(z,y):=\mathbb{E}_{(z,y)}\bigg[\int_0^{\infty}e^{-\rho s} H_{b}(Z_s,Y_s)ds\bigg],
\end{equation}
where 
\beq
\label{defHb}
H_{b}(z,y):=e^{z}C'(e^z)\mathds{1}_{\{y > b(z)\}} - \kappa (\delta - g - \rho - y)e^z \mathds{1}_{\{y \leq b(z)\}}, \quad (z,y) \in \mathbb{R}^2.
\eeq
Noticing that $|H_b(z,y)| \leq e^{z}C'(e^z) + \kappa e^z|\delta - g - \rho - y|$, by equation \eqref{ui1} of Lemma \ref{lem:limitZeui} one has that  $H_{b}(Z_s,Y_s) \in L^1(\mathbb{P}_{(z,y)} \otimes e^{-\rho s}ds)$ under Assumption \ref{ass:rho}. Due to this fact one can then verify that 
\begin{align}
\label{subharmonicw}
&\Big\{e^{-\rho t}w(Z^{z,y}_{t},Y^y_{t}) + \int_0^{t} e^{-\rho s}H_{b}(Z^{z,y}_s,Y^y_s)ds, \,t\geq 0\Big\}\,\,\text{is an $\mathcal{F}_t$-martingale.}
\end{align}
In particular, it is an $\mathcal{F}_t$-uniformly integrable martingale. Moreover, for any stopping time $\tau \in \mathcal{T}$
\begin{align}
\label{subharmonicwst}
& e^{-\rho \tau}w(Z^{z,y}_{\tau},Y^{y}_{\tau}) \leq \mathbb{E}_{(z,y)}\bigg[\int_0^{\infty} e^{-\rho s}\big|H_{b}(Z_s,Y_s)\big|ds\Big|\mathcal{F}_{\tau}\bigg],
\end{align}
and therefore the family of random variables $\{e^{-\rho \tau}w(Z^{z,y}_{\tau},Y^y_{\tau}),\,\tau \in \mathcal{T}\}$ is uniformly integrable.
Finally, in Lemma \ref{lemm:uleqw} in Appendix \ref{auxiliaryres} we show that $u \leq w$ on $\mathbb{R}^2$.\vspace{0.25cm}

\emph{Step 2.} We here prove that $\hat{y}(z) \leq b(z)$ for any $z \in \mathbb{R}$. Suppose that this is not true and that there exists a point $z_o \in \mathbb{R}$ such that $b(z_o) < \hat{y}(z_o)$. Then, taking $y < b(z_o)$ and setting $\sigma:=\sigma(z_o,y)=\inf\{t\geq 0:\, Y_t \geq \hat{y}(Z_t)\}$, $\mathbb{P}_{(z_o,y)}$-a.s., with regard to \eqref{def:w} and \eqref{subharmonicw} one finds
\begin{align}
\label{step3-1}
\mathbb{E}_{(z_o,y)}\Big[e^{-\rho\sigma}w(Z_{\sigma},Y_{\sigma})\Big] = w(z_o,y) - \mathbb{E}_{(z_o,y)}\bigg[\int_0^{\sigma} e^{-\rho s}H_{b}(Z_s,Y_s)ds\bigg].
\end{align}
The latter equality is due to an application of the optional stopping theorem (see Th.\ 3.2 in Ch.\ II of \cite{RY}), thanks to the uniform integrability of martingale \eqref{subharmonicw}.
On the other hand, \eqref{def:H} together with \eqref{subharmonic1} imply (cf.\ again Th.\ 3.2 in Ch.\ II of \cite{RY})
\begin{align}
\label{step3-2}
\mathbb{E}_{(z_o,y)}\Big[e^{-\rho\sigma}u(Z_{\sigma},Y_{\sigma})\Big] = u(z_o,y) - \mathbb{E}_{(z_o,y)}\bigg[\int_0^{\sigma} e^{-\rho s + Z_s} \kappa \big(\delta - g - \rho - Y_s\big) ds\bigg].
\end{align}
Substracting \eqref{step3-1} from \eqref{step3-2}, noticing that $w(z_o,y)=\kappa e^{z_o}=u(z_o,y)$ by \emph{Step 2} of Lemma \ref{lemm:uleqw} in Appendix \ref{auxiliaryres} since $y<b(z_o)<\hat{y}(z_o)$ by assumption, and considering that $u\leq w$ on $\mathbb{R}^2$ (cf.\ again Lemma \ref{lemm:uleqw} in Appendix \ref{auxiliaryres}) we find
\begin{align}
\label{step3-4}
& 0 \geq \mathbb{E}_{(z_o,y)}\Big[e^{-\rho\sigma}\big(u(Z_{\sigma},Y_{\sigma})-w(Z_{\sigma},Y_{\sigma})\big)\Big] \nonumber\\
&= \mathbb{E}_{(z_o,y)}\bigg[\int_0^{\sigma} e^{-\rho s}e^{Z_s}\Big(C'(e^{Z_s}) +\kappa(\delta - g - \rho - Y_s)\Big)\mathds{1}_{\{b(Z_s) < Y_s < \hat{y}(Z_s)\}} ds\bigg].
\end{align}
However, $\sigma >0$ $\mathbb{P}_{(z_o,y)}$-a.s.\ due to continuity of $(Z_{\cdot},Y_{\cdot})$ and continuity of $\hat{y}(\,\cdot\,)$, the set $\{(z,y) \in \mathbb{R}^2: b(z) < y < \hat{y}(z)\}$ is open and not empty because of the continuity of $\hat{y}(\,\cdot\,)$ and $b(\,\cdot\,)$, and $\hat{y}(z)\leq \frac{C'(e^z)}{\kappa} + \delta - g - \rho$. Hence the right-hand side of the latter is strictly positive and we reach a contradiction. Therefore, $\hat{y}(z) \leq b(z)$ for any $z \in \mathbb{R}$.
\vspace{0.25cm}

\emph{Step 3.} Here we prove that $\hat{y}(z) \geq b(z)$ for any $z \in \mathbb{R}$. Assume by contradiction that there exists a point $z_o \in \mathbb{R}$ such that $\hat{y}(z_o) < b(z_o)$. Take $y \in (\hat{y}(z_o), b(z_o))$ and consider the stopping time $\tau^{\star}:=\tau^{\star}(z_o,y)=\inf\{t\geq 0:\, Y_t \leq \hat{y}(Z_t)\}$, $\mathbb{P}_{(z_o,y)}$-a.s. This is optimal for $u(z_o,y)$. Then by \eqref{def:H} and \eqref{OSrepresent-new}, an application of Th.\ 3.2 in Ch.\ II of \cite{RY}
\begin{align}
\label{step4-1}
\mathbb{E}_{(z_o,y)}\Big[e^{-\rho\tau^{\star}}u(Z_{\tau^{\star}},Y_{\tau^{\star}})\Big] = u(z_o,y) - \mathbb{E}_{(z_o,y)}\bigg[\int_0^{\tau^{\star}} e^{-\rho s + Z_s}C'(e^{Z_s}) ds\bigg],
\end{align}
whereas by \eqref{def:w} and \eqref{subharmonicw} it follows
\begin{align}
\label{step4-2}
\mathbb{E}_{(z_o,y)}\Big[e^{-\rho\tau^{\star}}w(Z_{\tau^{\star}},Y_{\tau^{\star}})\Big] = w(z_o,y) - \mathbb{E}_{(z_o,y)}\bigg[\int_0^{\tau^{\star}} e^{-\rho s}H_{b}(Z_s,Y_s)ds\bigg].
\end{align}
Notice that we can also write 
\beq
\label{step4-3}
\mathbb{E}_{(z_o,y)}\Big[e^{-\rho\tau^{\star}}u(Z_{\tau^{\star}},Y_{\tau^{\star}})\Big] = \kappa \mathbb{E}_{(z_o,y)}\Big[e^{-\rho\tau^{\star} + Z_{\tau^{\star}}}\Big] = \mathbb{E}_{(z_o,y)}\Big[e^{-\rho\tau^{\star}}w(Z_{\tau^{\star}},Y_{\tau^{\star}})\Big].
\eeq
Here the first equality follows from the uniform integrability of $\{e^{-\rho\tau}u(Z_{\tau},Y_{\tau}), \tau \in \mathcal{T}\}$, together with standard localisation arguments, and the optimality of $\tau^{\star}$ for $u(z_o,y)$; the second equality is due to \emph{Step 2} of Lemma \ref{lemm:uleqw} in Appendix \ref{auxiliaryres} (since by \emph{Step 2} of this proof we already know that $\hat{y}(\,\cdot\,) \leq b(\,\cdot\,)$), and again to a localisation argument exploiting the uniform integrability of $\{e^{-\rho\tau}w(Z_{\tau},Y_{\tau}), \tau \in \mathcal{T}\}$.

Then, substracting \eqref{step4-2} from \eqref{step4-1}, and taking into account \eqref{step4-3} and that $u \leq w$ on $\mathbb{R}^2$ (cf.\ Lemma \ref{lemm:uleqw} in Appendix \ref{auxiliaryres}), we find
\begin{align}
\label{step4-4}
& 0 \geq \mathbb{E}_{(z_o,y)}\bigg[\int_0^{\tau^{\star}} e^{-\rho s}e^{Z_s}\Big(C'(e^{Z_s}) +\kappa(\delta - g - \rho - Y_s)\Big)\mathds{1}_{\{\hat{y}(Z_s) < Y_s \leq b(Z_s)\}} ds\bigg].
\end{align}
Now $\tau^{\star} >0$ $\mathbb{P}_{(z_o,y)}$-a.s.\ by continuity of $(Z_{\cdot},Y_{\cdot})$ and of $\hat{y}(\,\cdot\,)$. Moreover, the set $\{(z,y) \in \mathbb{R}^2: \hat{y}(z) < y \leq b(z)\}$ is open and not empty because of the continuity of $\hat{y}(\,\cdot\,)$ and $b(\,\cdot\,)$, and $b(z)\leq \frac{C'(e^z)}{\kappa} + \delta - g - \rho$ since $b \in \mathcal{M}$ by assumption. The right-hand side of \eqref{step4-4} is therefore strictly positive and we reach a contradiction; that is, $\hat{y}(z) \geq b(z)$ for any $z \in \mathbb{R}$.
\vspace{0.25cm}

\emph{Step 4.} By \emph{Step 2} and \emph{Step 3} we conclude that $\hat{y} \equiv b$ on $\mathbb{R}$ and we thus complete the proof.
\end{proof}

\begin{remark}
Defining
$$K(z,\xi,v,\alpha):=\int_0^{\infty}e^{-\rho t}\bigg(\int_{\alpha}^{\infty} e^v\Big(C'(e^v) + \kappa(\delta-g-\rho-u)\Big) p_t(z,\xi;v,u)du\bigg)dt$$
and 
$$f(z,\xi):=\int_0^{\infty}e^{-\rho t}\bigg(\int_{\mathbb{R}^2} \kappa e^v(\delta-g-\rho-u) p_t(z,\xi;v,u)du\,dv\bigg)dt,$$ 
integral equation \eqref{inteq-2} can be rewritten by Fubini's theorem as
$$\kappa e^z + f(z,f(z)) = \int_{\mathbb{R}} K(z,f(z),v,f(v)) dv.$$
The latter reformulation shows that \eqref{inteq-2} belongs to the class of nonlinear Fredholm equations of second kind (see, e.g., \cite{Delves} or \cite{Hackbusch}). Since the state space of $(Z,Y)$ is unbounded, \eqref{inteq-2} is actually a so-called {\rm{singular nonlinear Fredholm equation of second kind}}. Classical textbooks offering a survey of numerical methods for equations of this kind are \cite{Baker} and \cite{Delves}, among others. However, being any of these methods certainly non trivial, we believe that such numerical computation falls outside the scopes of our work.
\end{remark}

\subsubsection{Asymptotic Behaviour of the Free Boundary}
\label{sec:asymptotic}

Here we provide the asymptotic behaviour of the free boundary $\hat{y}$. The results of this section will allow us to obtain in Section \ref{economics} interesting economic considerations on the optimal debt reduction policy (see in particular (ii) in Section \ref{economics}).

To perform our analysis we need to introduce the one-dimensional optimal stopping problem with value 
\beq
\label{OS-1dim}
w(y) := \inf_{\tau \geq 0}\mathbb{E}\bigg[\int_0^{\tau} e^{-\rho t + (\delta - g) t - \int_0^t Y^y_s ds} C'(0) dt + \kappa\Big(e^{-\rho \tau + (\delta - g) \tau - \int_0^{\tau} Y^y_s ds} - 1 \Big)dt \bigg], \quad y \in \mathbb{R}.
\eeq
Notice that the process $Y$ appears in \eqref{OS-1dim} only in the discount factor through its time time-integral. However, being $Y$ real-valued, such random discounting cannot be seen as an additive functional of $Y$, and techniques from, e.g., \cite{Dayanik} cannot be applied. The following result characterises the optimal stopping rule for problem \eqref{OS-1dim}. Its proof can be found in Appendix \ref{someproofs} and is based on a direct probabilistic analysis of the value function $w$.
\begin{proposition}
\label{prop:1dimOS}
For $w$ as in \eqref{OS-1dim} let $y^*:= \sup\{y \in \mathbb{R}: w(y) \geq 0\}$, with the convention $\sup\emptyset = - \infty$, and set $\sigma^*(y):=\inf\{t\geq 0: Y^y_t \leq y^*\}$, $y \in \mathbb{R}$, with the convention $\inf\emptyset = + \infty$. The following holds true:
\begin{itemize}
\item[(i)] if $C'(0) > 0$, then $y^*$ is given by the unique $y \in (-\infty, \delta - g- \rho)$ solving 
\beq
\label{inteq-ystar}
\mathbb{E}\bigg[\int_0^{\infty} e^{-\rho t + (\delta - g) t - \int_0^t Y^y_s ds}\Big(C'(0) + \kappa(\delta - g - \rho - Y^y_t)\Big)\mathds{1}_{\{Y^y_t > y\}} dt\bigg] = 0,
\eeq
and the a.s.\ finite stopping time $\sigma^*(y):=\inf\{t\geq 0: Y^y_t \leq y^*\}$, $y \in \mathbb{R}$, is optimal for \eqref{OS-1dim};
\item[(ii)] If $C'(0) = 0$, then $y^* = -\infty$ and  $\sigma^*(y) = +\infty$, $y \in \mathbb{R}$. In particular, $w(y) = -\kappa$ for any $y\in \mathbb{R}$.
\end{itemize}
\end{proposition}


\begin{proposition}
\label{prop:limits}
For $y^*$ as given in Proposition \ref{prop:1dimOS} one has:
\begin{itemize}
\item[(i)] $\lim_{z \uparrow +\infty} \hat{y}(z) = \infty$;
\item[(ii)] $\lim_{z \downarrow -\infty} \hat{y}(z) = y^*$.
\end{itemize}
\end{proposition}
\begin{proof}

We prove each item separately.
\vspace{0.15cm}

\emph{(i)}\,\,We follow a contradiction scheme. Set $\lim_{z \uparrow \infty} \hat{y}(z):= \hat{y}_{\infty}$, which exists by monotonicity, and suppose $\hat{y}_{\infty} < \infty$. Then take any $z\in \mathbb{R}$, $y_2 > y_1 > \hat{y}_{\infty}$, and for arbitrary $y \in (y_1,y_2)$ set $\tau_o(y):=\inf\{t\geq 0: Y^y_t \notin (y_1,y_2)\}$ a.s. Notice now that $\tau_o(y) \leq \tau^{\star}(z,y)$ a.s.\ (cf.\ \eqref{OStime}) and $(z,y) \in \mathcal{C}$, and employ the subharmonic characterisation of the value function \eqref{valueOS}, together with Th.\ 3.2 in Ch.\ II of \cite{RY} due to \eqref{ui1}, so to write
\begin{eqnarray}
\label{limit1}
&& \kappa e^z > u(z,y) 
= \mathbb{E}\bigg[\int_0^{\tau_o(y)} e^{-\rho t + Z^{z,y}_t} C'( e^{Z^{z,y}_t}) dt + e^{-\rho \tau_o(y)}u(Z^{z,y}_{\tau_o(y)},Y^y_{\tau_o(y)})\bigg] \\
&& \geq \mathbb{E}\bigg[\int_0^{\tau_o(y)} e^{-\rho t + Z^{z,y}_t} C'( e^{Z^{z,y}_t}) dt\bigg] \geq \mathbb{E}\bigg[\int_0^{\tau_o(y)} e^{-\rho t} C( e^{Z^{z,y}_t}) dt\bigg].\nonumber
\end{eqnarray}
Here the fourth step is due to the nonnegativity of $u$, whereas the last one follows by convexity of $C$ and the fact that $C(0)=0$. By using now Assumption \ref{ass:c}-\emph{(ii)} we get from \eqref{limit1}
\beq
\label{limit2}
\kappa > e^{(\gamma - 1) z} K_1 \mathbb{E}\bigg[\int_0^{\tau_o(y)} e^{-\rho t + \gamma Z^{0,y}_t} dt\bigg] - \frac{K e^{-z}}{\rho}\Big(1 - \mathbb{E}\big[e^{-\rho \tau_o(y)}]\Big),
\eeq
for any $y \in (y_1,y_2)$ and for any $z\in \mathbb{R}$. Recalling that $\gamma>1$ we then reach a contradiction by taking $z$ sufficiently large, and we therefore conclude that $\lim_{z \uparrow \infty} \hat{y}(z) = \infty$.
\vspace{0.08cm}

\emph{(ii)}\,\,The proof of this last property is more involved, and it is organised in three steps. 
\vspace{0.08cm}

\emph{Step 1.}\,\,Let $y^*$ as defined in Proposition \ref{prop:1dimOS}. Here we show that $\lim_{z \uparrow -\infty} \hat{y}(z) \geq y^*$. This is clearly true if $C'(0)=0$, hence we now consider only the case $C'(0)>0$. Notice that from \eqref{valueOS} we can write
\begin{eqnarray}
\label{limitesotto1}
&& e^{-z}(u(z,y) - \kappa e^z) = \inf_{\tau \geq 0}\mathbb{E}\bigg[\int_0^{\tau} e^{-\rho t + Z^{0,y}_t} C'(e^{Z^{z,y}_t}) dt + \kappa\big(e^{-\rho \tau + Z^{0,y}_{\tau}} - 1 \big)dt \bigg]\nonumber \\
&& \geq \inf_{\tau \geq 0}\mathbb{E}\bigg[\int_0^{\tau} e^{-\rho t + Z^{0,y}_t} C'(0) dt + \kappa\big(e^{- \rho\tau + Z^{0,y}_{\tau}} - 1 \big)dt \bigg] = w(y),
\end{eqnarray}
where we have used that $C'(\,\cdot\,)$ in nondecreasing by convexity of $C(\,\cdot\,)$. It thus follows that for any $z\in \mathbb{R}$ the $z$-section of $\mathcal{S}$, i.e.\ $\mathcal{S}_z:=\{y \in \mathbb{R}: u(z,y)=\kappa e^z\}$ is such that $\mathcal{S}_z\supseteq \{y \in \mathbb{R}: w(y)=0\}$. Because by Proposition \ref{prop:1dimOS} $\{y \in \mathbb{R}: w(y)=0\} = \{y \in \mathbb{R}: y \leq y^*\}$, we conclude that $\hat{y}(z) \geq y^*$ for any $z\in \mathbb{R}$, and therefore that $\lim_{z \downarrow -\infty} \hat{y}(z) \geq y^*$.
\vspace{0.08cm}

\emph{Step 2.}\,\,We now prove that $\lim_{z\downarrow -\infty} e^{-z}(u(z,y) - \kappa e^z) = w(y)$. By \eqref{limitesotto1} (which holds if $C'(0) \geq 0$) we clearly have $\liminf_{z\downarrow -\infty}e^{-z}(u(z,y) - \kappa e^z) \geq w(y)$. On the other hand, recalling that $\sigma^*(y)=\inf\{t\geq 0: Y^y_t \leq y^*\}$ is the optimal stopping time for $w(y)$ (with the convention $\inf\emptyset=+\infty$) we can write
\begin{eqnarray}
\label{limitesotto2}
&& e^{-z}(u(z,y) - \kappa e^z) \leq \mathbb{E}\bigg[\int_0^{\sigma^*(y)} e^{-\rho t + Z^{0,y}_t} C'(e^{z + Z^{0,y}_{t}}) dt +\kappa\big(e^{-\rho \sigma^*(y) + Z^{0,y}_{\sigma^*(y)}} - 1\big)\bigg].
\end{eqnarray}
Taking limit superior as $z\downarrow -\infty$ on both sides of \eqref{limitesotto2}, invoking the monotone convergence theorem we find $\limsup_{z\downarrow -\infty} e^{-z}(u(z,y) - \kappa e^z) \leq w(y)$, which finally implies $\lim_{z\downarrow -\infty} e^{-z}(u(z,y) - \kappa e^z) = w(y)$.
\vspace{0.08cm}

\emph{Step 3.}\,\,Here we prove that $\hat{y}_{-\infty}:=\lim_{z \downarrow -\infty} \hat{y}(z)= y^*$. Since we already know by \emph{Step 1} that $\hat{y}_{-\infty} \geq y^*$, we argue by contradiction and we suppose that $\hat{y}_{-\infty} > y^*$. Then take $y_1,y_2$ such that $y^* < y_1 < y_2 < \hat{y}_{-\infty}$, set $\mathcal{Q}:=\{ y \in \mathbb{R}:\, y \in (y_1,y_2)\}$ and note that $\mathcal{Q} \subseteq \mathcal{C}_w = \{y \in \mathbb{R}: w(y) <0\}$. Then noticing that $\{(z,y) \in \mathbb{R}^2:\, y \in \mathcal{Q}\} \subset \mathcal{S}$ we have by \emph{Step 2}
\beq
\label{limitesotto3}
0= \lim_{z \downarrow -\infty}e^{-z}(u(z,y) - \kappa e^z) = w(y), \quad y \in \mathcal{Q}.
\eeq
However, $w(y)=0$ is not possible on $\mathcal{Q}$, because $\mathcal{Q} \subseteq \mathcal{C}_w$. Hence we reach a contradiction, and we have thus proved that $\lim_{z \downarrow -\infty} \hat{y}(z)= y^*$.
\end{proof}


\subsection{The Optimal Cumulative Primary Balance and the Minimal Cost}
\label{sec:OC}

In this section we provide the optimal debt reduction policy. This takes the form of a threshold policy, where the boundary triggering the optimal intervention rule is closely related to the stopping boundary $\hat{y}$ completely characterised in the last section.

\subsubsection{The Action and Inaction Regions}

Recall that $\mathcal{O}:=(0,\infty) \times \mathbb{R}$ and define
\begin{align}
\label{regionsCA0}
\mathcal{I}:=\{(x,y)\in \mathcal{O}\ | \ u(\ln(x),y)< x\}\quad\text{and}\quad \mathcal{I}^c:=\{(x,y)\in \mathcal{O} \ | \ u(\ln(x),y) = x\}.
\end{align}
The sets $\mathcal{I}$ and $\mathcal{I}^c$ are respectively the candidate \emph{inaction} region and the candidate \emph{action} region for the control problem \eqref{eq:valueOC}. When the state variable belongs to $\mathcal{I}^c$ it should be optimal to reduce the level of debt ratio, as it is too high. On the other hand, a non intervention policy should be applied in region $\mathcal{I}$. This will be verified in the following. 

Thanks to Proposition \ref{continuityOS}, $\mathcal{I}$ and $\mathcal{I}^c$ are open and closed, respectively. Moreover, it is clear that they can be expressed also as
$\mathcal{I}=\{(x,y)\in \mathcal{O}\ | \ \frac{1}{x}u(\ln(x),y)< 1\}$ and $\mathcal{I}^c =\{(x,y)\in \mathcal{O} \ | \ \frac{1}{x}u(\ln(x),y) = 1\}$, where a simple calculation from \eqref{valueOS} reveals that
\begin{equation}
\label{def:1suxu}
\frac{1}{x}u(\ln(x),y) = \inf_{\tau \geq 0}\mathbb{E}\bigg[\int_0^{\tau} e^{-\rho t + (\delta - g)t -\int_0^t Y^y_s ds}C'(x e^{(\delta - g)t -\int_0^t Y^y_s ds}) + e^{-\rho \tau + (\delta - g)\tau -\int_0^{\tau} Y^y_s ds}\bigg],
\end{equation}
for any $(x,y)\in\mathcal{O}$.

It follows from \eqref{def:1suxu} that $y \mapsto \frac{1}{x}u(\ln(x),y)$ is nonincreasing and $x \mapsto \frac{1}{x}u(\ln(x),y)$ is nondecreasing. The latter property implies that for fixed $y \in \mathbb{R}$ the region $\mathcal{I}$ is below $\mathcal{I}^c$, and we define the boundary between these two regions by
\beq
\label{boundary-b}
b(y):=\sup\{x > 0:\, u(\ln(x),y) < x\}, \qquad y \in \mathbb{R},
\eeq
with the convention $\sup \emptyset = 0$. Then $\mathcal{I}$ and $\mathcal{I}^c$ can be equivalently written as
\begin{align}
\label{regionsCA1}
\mathcal{I}=\{(x,y)\in \mathcal{O}\ | \ x < b(y)\}\quad\text{and}\quad \mathcal{I}^c=\{(x,y)\in \mathcal{O} \ | \ x \geq b(y)\}.
\end{align}
From the previous and from \eqref{cont}, \eqref{stopp} and \eqref{contstopp} we also have
\begin{eqnarray}
\label{equivalences}
e^z \ < \ b(y)\ \Longleftrightarrow  \ u(z,y) \ < \ e^z \ \Longleftrightarrow \ y \ > \ \hat{y}(z), \quad (z,y) \in \mathbb{R}^2.
\end{eqnarray}
Hence, for any $y \in \mathbb{R}$, $b$ of \eqref{boundary-b} can be seen as the pseudo-inverse of the nondecreasing (cf.\ Proposition \ref{prop:preliminarybdy}) function $z \mapsto \hat{y}(z)$ composed with the logarithmic function; that is,
\beq
\label{bhaty}
b(y)=\sup\{ x > 0 \ |\ y> \hat{y}(\ln(x))\}, \quad y \in \mathbb{R}.
\eeq
It thus follows that the characterization of $\hat{y}$ of Theorem \ref{th:integral} is actually equivalent to a complete characterization of $b$ thanks to \eqref{bhaty}.

Recall $y^*$ of Proposition \ref{prop:1dimOS}. The next proposition collects some properties of $b$, and its proof can be found in Appendix \ref{someproofs}.
\begin{proposition}
\label{propertiesb}
The boundary $b$ of \eqref{boundary-b} is such that
\begin{enumerate}
\item $y \mapsto b(y)$ is nondecreasing and left-continuous;
\item $C'(b(y))\geq \kappa[y - \delta + g + \rho]^+$;
\item $b(y) =0$ for any $y\leq y^*$, whenever $y^* > -\infty$.  
\end{enumerate}
\end{proposition}

\subsubsection{The Optimal Control}

For $b$ as in \eqref{boundary-b} (see also \eqref{bhaty}) introduce the nondecreasing process
\beq
\label{OC-candidate}
\displaystyle \overline{\nu}^{\star}_t = \Big[ x - \inf_{0\leq s \leq t}\Big(b(Y^y_s)e^{-(\delta - g)s + \int_0^s Y^y_u du}\Big)\Big] \vee 0, \quad t \geq 0, \qquad \overline{\nu}^{\star}_{0-}=0,
\eeq
and then the process
\beq  
\label{nustar}
\nu^{\star}_t:=\int_{0}^t e^{(\delta - g)s - \int_0^s Y^y_u du} d\overline{\nu}^{\star}_s, \quad t \geq 0, \qquad \nu^{\star}_{0-}=0. 
\eeq

\begin{theorem}
\label{teo:ver-2}
Let $\widehat{v}(x,y):= \int_{-\infty}^{\ln(x)}u(q,y)dq$, $(x,y)\in \mathcal{O}$. Then one has $\widehat{v} = v$ on $\mathcal{O}$, and $\nu^{\star}$ as in \eqref{nustar} is optimal for the control problem \eqref{eq:valueOC}.
\end{theorem}

\begin{proof}
By the proof of Theorem \ref{teo:ver} (upon noticing that in our Markovian setting one has $u=U$ and $\widehat{v}=\widehat{V}$) it suffices to show that the right-continuous inverse of the stopping time $\tau^{\star}(q,y)= \ \inf\{t\geq 0 \ | \ Y^y_t \leq \hat{y}(Z^{q,y}_t) \}$ (which is optimal for $u(q,y)$, cf.\ \eqref{OStime}) coincides (up to a null set) with $\overline{\nu}^{\star}$.

Then, recall \eqref{taunumeno} from the proof of Theorem \ref{teo:ver}, fix $(x,y)\in\mathcal{O}$, take $t\ge 0$ arbitrary, and note that by \eqref{OStime} and \eqref{equivalences} we have $\mathbb{P}$-a.s.~the equivalences
\begin{align*}
&\tau^{\star}(q,y) \leq t  \  \Longleftrightarrow \ Y^y_{\theta} \leq \hat{y}(Z^{q,y}_{\theta}) \ \mbox{ for some } \theta \in [0,t] \ \Longleftrightarrow \\
& b(Y^y_{\theta}) \leq e^{q + (\delta - g)\theta - \int_0^{\theta}Y^y_sds} \ \mbox{ for some } \theta \in [0,t] \ \Longleftrightarrow\   \Big[ x - \inf_{0\leq s \leq t}\Big(b(Y^y_s)e^{-(\delta - g)s + \int_0^s Y^y_u du}\Big)\Big] \vee 0 \geq x - e^q \\
& \Longleftrightarrow\ \overline{\nu}^{\star}_t \geq x - e^q \Longleftrightarrow\ \tau^{\overline{\nu}^{\star}}_{+}(q) \leq t.
\end{align*}
Hence, $\tau^{\overline{\nu}^{\star}}_{+}(q)=\tau^{\star}(q,y)$ $\mathbb{P}$-a.s., and $\overline{\nu}^{\star}_{\cdot}$ is the right-continuous inverse of $\tau^{\star}(\cdot,y)$. Since $\overline{\nu}^{\star}_t \leq x$ a.s.\ for all $t\geq 0$, and $t \mapsto \overline{\nu}^{\star}_t$ is nondecreasing, it does follows from \eqref{nustar} that $\nu^{\star}$ is admissible. By arguing as in \emph{Step 2} of the proof of Theorem \ref{teo:ver} the claim follows.


\end{proof}


\section{Economic Conclusions and Related Problems}
\label{sec:optimalsol}

\subsection{Some Economic Conclusions}
\label{economics}

In this section we provide some comments about the economic implications of our findings.

From the first equation of \eqref{syst:XY}, and recalling \eqref{nustar}, we can write 
$$X^{x,y,\nu^{\star}}_t = e^{(\delta-g)t - \int_0^t Y_s ds}\big[x - \overline{\nu}^{\star}_t\big],$$
which, with regard to \eqref{OC-candidate}, shows that 
\begin{equation}
\label{cond1}
0 \leq X^{x,y,\nu^{\star}}_t \leq b(Y^y_t), \qquad t \geq 0,\,\,\mathbb{P}-a.s.
\eeq
Moreover, it is easy to see that we can express $\overline{\nu}^{\star}$ of \eqref{OC-candidate} as
\begin{equation}
\label{cond2}
\overline{\nu}^{\star}_t = \sup_{0\leq u \leq t} \Big(\frac{X^{x,y,0}_s - b(Y^y_s)}{X^{1,y,0}_s}\Big) \vee 0, \qquad \overline{\nu}^{\star}_{0-}=0.
\end{equation}

Equations \eqref{OC-candidate}, \eqref{cond1} and \eqref{cond2} allow us to draw some interesting conclusions about the optimal debt management policy suggested by our model.
\begin{itemize}
\item[(i)] If at initial time the level of the debt ratio $x$ is above $b(y)$, then an immediate lump-sum reduction of amplitude $(x-b(y))$ is optimal. 
\item[(ii)] When the cost of any additional unit of debt to the zero level is strictly positive (i.e.\ $C'(0)>0$), and the initial level of the inflation rate is sufficiently low (namely, $y \leq y^*$, $y^* > -\infty$), the optimal action is to immediately reduce the debt ratio to zero (cf.\ Proposition \ref{propertiesb}-(3)). This fact might be explained as follows. Negatively affecting the growth rate of the debt ratio, a low inflation increases the real burden of debts, and thus have bad effects on the economy. Then, if the economy faces low inflation (or deflation), and if the country is penalised by any, even small, level of debt\footnote{indeed, by convexity, if $C'(0)>0$ then $C(\varepsilon) \geq C'(0) \varepsilon > 0$, for any $\varepsilon>0$.}, according to our model it is optimal to immediately bring the debt ratio to zero, rather than waiting for an higher inflation rate.
\item[(iii)] At any $t\geq 0$, employing the optimal primary balance policy, the government keeps the debt ratio level below the inflation-dependent ceiling $b$. 
\item[(iv)] If the level of the debt ratio at time $t$ is below $b(Y_t)$, there is no need for interventions. The government should intervene to reduce its debt generating fiscal surpluses only at those (random) times $t$ for which $X_t \geq b(Y_t)$, any other intervention being sub-optimal (cf.\ \eqref{cond2}).
\item[(v)] The optimal cumulative debt reduction \eqref{OC-candidate} is given in terms of the minimal level that the inflation-dependent debt ceiling, suitably discounted, has reached over $[0,t]$. The discount rate is dynamic and does depend on the history of inflation. It is indeed $\int_0^{\cdot} (\delta-g - Y^y_u)du$.
\end{itemize}


\subsection{Two Related Classes of Problems}
\label{sec:relatedprobl}

Problem \eqref{eq:valueOC-general} shares a common mathematical structure with other optimisation problems arising in economic theory/mathematical finance, like problems of irreversible investment \cite{DP} with (stochastic) depreciation in the capital dynamics \cite{RS}, and optimal consumption choices under Hindy-Huang-Kreps (HHK) preferences (cf.\ \cite{BankRiedel1}, \cite{HH} and references therein). In this section we provide an informal brief introduction to these problems, so to highlight their similarities with the problem treated in this paper.

Let $(\Omega, \mathcal{F}, \mathbb{F}, \mathbb{P})$ be a given complete filtered probability space. In \emph{stochastic irreversible investment problems}, the manager of a firm producing a single good aims at expanding the company's production capacity, e.g., in order to meet an exogenous demand of the produced good. However, increasing the production capacity level, the firm incurs costs that are proportional to the amount of capacity expansion that has been made. Mathematically, the problem amounts to find an $\mathbb{F}$-adapted nondecreasing c\`adl\`ag process $I^{\star}$ minimising the total expected cost functional
\beq
\label{Irr}
\mathcal{J}(I):=\mathbb{E}\bigg[\int_0^{\infty} e^{-\rho t} f(D_t, X^{I}_t)dt + \kappa \int_0^{\infty}e^{-\rho t} dI_t\bigg].
\eeq
In \eqref{Irr} $f:\mathbb{R}\times \mathbb{R} \mapsto \mathbb{R}_+$ is a suitable convex instantaneous cost function; $D:=\{D_t, t \geq 0\}$ is the random demand of the produced good; $\kappa>0$ is the marginal cost of investment; $\rho>0$ is the manager's discount factor, and $X^{I}_t$ denotes the level of production capacity at time $t$ when the irreversible investment plan $I$ is followed. Assuming that the capital depreciates at a stochastic depreciation rate $\eta:=\{\eta_t, t \geq 0\}$, the dynamics of the production capacity can be modelled as
\beq
\label{capacity}
dX^{I}_t = -\eta_t X^{I}_t dt + dI_t, \quad t\geq 0, \qquad X^{\nu}_{0-}=x>0.
\eeq
Setting $h(\omega,t,x) := e^{-\rho t} f(D_t(\omega), x)$ and $\kappa_t:=\kappa e^{-\rho t}$, it is clear that the problem of minimising \eqref{Irr} over all suitable irreversible investment plans $I$'s, shares a similar structure with problem \eqref{eq:valueOC-general}. In fact, the only difference lies in the fact that in \eqref{capacity} the production capacity is increased, whereas in \eqref{dynX} the level of debt ratio must be reduced.

In the literature on \emph{optimal consumption with HHK preferences} (see \cite{BankRiedel1} for details and references), an economic agent aims at maximising her intertemporal expected utility by investing an initial wealth $w \geq 0$ in the financial market over a given time period. Following \cite{BankRiedel1}, we suppose that the agent can invest in at least one risky security and in a money market account with interest rate $r$. Assuming further that the market is complete, and denoting by $\psi:=\{\psi_t, t \geq 0\}$ a state-price density, the optimal consumption plan $C^{\star}$ (if it does exist) then solves the problem
\beq
\label{HHK}
\sup_{C}\mathbb{E}\bigg[\int_0^{\infty} U(t, X^C_t) dt\bigg],
\eeq
where the optimisation is taken over all the $\mathbb{F}$-adapted, nondecreasing and c\`adl\`ag processes such that $\mathbb{E}[\int_0^{\infty} \psi_t dC_t]\leq w$. Here $\mathbb{E}[\int_0^{\infty} \psi_t dC_t]$ is the minimal capital needed to finance a given consumption plan $C$. In \eqref{HHK} the instantaneous felicity function $U(t,\cdot)$ is a suitable concave function of the agent's current level of satisfaction
\beq
\label{satisfaction}
dX^{C}_t = -\vartheta_t X^{C}_t dt + dC_t, \quad t\geq 0, \qquad X^{C}_{0-}=x>0,
\eeq
for some process $\vartheta:=\{\vartheta_t, t\geq 0\}$ measuring the decay rate of satisfaction.
As we have already noticed for \eqref{capacity}, the main difference between the controlled dynamics \eqref{satisfaction} and \eqref{dynX} is that in \eqref{satisfaction} the level of satisfaction is increased by consumption, whereas in \eqref{dynX} the level of debt ratio is reduced.
In \cite{BankRiedel1} it is shown that the Lagrangian functional associated to such problem reads
\beq
\label{Lagrangian}
\mathcal{L}(C):=\mathbb{E}\bigg[\int_0^{\infty} U(t, X^C_t) dt\bigg] - \lambda\bigg(\mathbb{E}\bigg[\int_0^{\infty} \psi_t dC_t\bigg] - w\bigg),
\eeq
for a suitable Lagrange multiplier $\lambda \geq 0$. Given $\lambda>0$ and setting $h(\omega,t,x) := - U(t,x)$ and $\kappa_t:= - \lambda\psi_t$, maximising \eqref{Lagrangian} over all admissible irreversible consumption plans $C$'s relates to the minimisation problem \eqref{eq:valueOC-general}.

From the previous discussion it thus follows that the approach that we have followed in this paper might be suitably adopted to solve problems of optimal consumption choice under HHK preferences, and problems of stochastic irreversible investment plans with depreciating capital stock. In particular, once these problems are formulated in suitable Markovian settings, their optimal investment/consumption policy might be expressed in terms of the free boundary of an associated optimal stopping problem. We leave such interesting study for future research.

\bigskip

\textbf{Acknowledgments.} I thank the Associate Editor and three anonymous Referees for their pertinent comments which helped a lot to improve previous versions of this work. I also thank Tiziano De Angelis and Frank Riedel for useful discussions.


\appendix
\section{Some Proofs}
\label{someproofs}
\renewcommand{\theequation}{A-\arabic{equation}}

\emph{Proof of Proposition \ref{prop:verificoipotesi}}
\vspace{0.15cm}

The proof is organised in two steps.
\vspace{0.15cm}

\emph{Step 1.} We start showing that the deterministic process $\kappa_t:=\kappa\, e^{-\rho t}$, $\kappa>0$, satisfies Assumption \ref{ass:kappa}. It is clearly progressively measurable (being deterministic), positive and with continuous paths. Let $(x,y) \in (0,\infty) \times \mathbb{R}$ be given and fixed and recall $X^{x,y,0}$ from \eqref{syst:XY}. It then remains only to prove that the process $\{\kappa e^{-\rho t} X^{x,y,0}_t, t\geq 0\}$ is of class (D).

Thanks to an integration by parts, for any $t\geq 0$ one has 
$$\kappa e^{-\rho t} X^{x,y,0}_t = \kappa x e^{-\rho t + (\delta-g)t - \int_0^s Y^y_u du } = \kappa x \Big(1 + \int_{0}^{t} e^{-\rho t + (\delta-g)t - \int_0^s Y^y_u du }\big(\delta - g - Y^y_s - \rho\big) ds\Big), \quad \mathbb{P}-a.s.,$$
so that if $\mathbb{E}[\int_{0}^{\infty} e^{-\rho t + (\delta-g)t - \int_0^s Y^y_u du }\big|\delta - g - Y^y_s - \rho\big| ds] < \infty$ then the process $\{\kappa e^{-\rho t} X^{x,y,0}_t, t\geq 0\}$ is of class (D).

By \eqref{av-intY} and \eqref{var-intY}, simple estimates and Assumption \ref{ass:rho} imply that there exists $0<K_1(y)<\infty$ such that 
$$\mathbb{E}\bigg[\int_0^{\infty} e^{-\rho s + 2(\delta-g)t - 2\int_0^t Y^y_u du} ds\bigg] \leq K_1(y).$$ Then H\"older inequality with respect the product measure $\mathbb{P} \otimes e^{-\rho t} dt$ gives
\begin{align*}
\label{ui3}
& \mathbb{E}\bigg[\int_0^{\infty} e^{-\rho t + (\delta-g)t - \int_0^t Y^y_u du}|\delta - g - \rho - Y^y_t| dt\bigg] \leq \mathbb{E}\bigg[\int_0^{\infty} e^{-\rho t + 2(\delta-g)t - 2\int_0^t Y^y_u du} dt\bigg]^{\frac{1}{2}} \times \nonumber \\
& \mathbb{E}\bigg[\int_0^{\infty} e^{-\rho t}|\delta - g -\rho - Y^y_t|^2 dt\bigg]^{\frac{1}{2}} \leq
\sqrt{K_1(y)}\mathbb{E}\bigg[\int_0^{\infty} e^{-\rho t}|\delta - g -\rho - Y^y_t|^2 dt\bigg]^{\frac{1}{2}}.
\end{align*}
Since by the second equation of \eqref{syst:XY} one has $\mathbb{E}[\int_0^{\infty} e^{-\rho t}|\delta - g -\rho - Y^y_t|^2 dt] \leq K_2(1+ |y|^2)$, for some $K_2>0$, the proof is complete.

\vspace{0.15cm}

\emph{Step 2.} We now prove that if $C$ satisfies Assumption \ref{ass:c}, then $h(\omega,t,x):=e^{-\rho t}C(x)$ fulfills Assumption \ref{ass:h}. It clearly verifies \emph{(i)} and \emph{(ii)} of Assumption \ref{ass:h}. We now prove that \emph{(iii)} of Assumption \ref{ass:h} holds true as well.

Let $(x,y) \in (0,\infty) \times \mathbb{R}$ be given and fixed. By Assumption \ref{ass:c} one has
\begin{equation}
\label{eq:verificoipotesi-1}
\mathbb{E}_{(x,y)}\bigg[\int_0^{\infty} h(t,X^{0}_t) dt\bigg] =\mathbb{E}_{(x,y)}\bigg[\int_0^{\infty} e^{-\rho t} C(X^0_t) dt\bigg] \leq \int_0^{\infty} e^{-\rho t} K\Big(1 + \mathbb{E}_{(x,y)}\big[(X^0_t)^{\gamma}\big]\Big) dt,
\end{equation}
where the last step follows by Tonelli's Theorem. To evaluate the expectation in the right-hand side of \eqref{eq:verificoipotesi-1} notice that by \eqref{syst:XY} we can write for any $t\geq 0$
\begin{align}
\label{eq:verificoipotesi-2}
& \mathbb{E}_{(x,y)}\big[(X^0_t)^{\gamma}\big] = x^{\gamma} \exp\Big\{ \gamma(\delta - g)t -\gamma\mathbb{E}\Big[\int_0^t Y^y_u du\Big] + \frac{\gamma^2}{2}\Var\Big[\int_0^t Y^y_u du\Big]\Big\},
\end{align}
where the formula of the Laplace transform for the Gaussian random variable $\int_0^t Y^y_u du$ has been employed. Plugging \eqref{av-intY} and \eqref{var-intY} of Lemma \ref{lemm:Y} into \eqref{eq:verificoipotesi-2}, simple algebra and standard inequalities yield
\begin{equation}
\label{eq:verificoipotesi-3}
\mathbb{E}_{(x,y)}\big[(X^0_t)^{\gamma}\big] \leq x^{\gamma} e^{\frac{\gamma}{\theta}|y-\frac{a}{\theta}|}\exp\Big\{-\gamma t\Big(\frac{a}{\theta} - \frac{\gamma \sigma^2}{2 \theta^2} - \delta + g\Big)\Big\}.
\end{equation}
Employing the inequality above in \eqref{eq:verificoipotesi-1}, by Assumption \ref{ass:rho} we conclude $\mathbb{E}_{(x,y)}[\int_0^{\infty} h(t,X^{0}_t) dt] <\infty$. Analogously, by Assumption \ref{ass:c} we have that
\begin{eqnarray}
\label{eq:verificoipotesi-4}
&& \mathbb{E}_{(x,y)}\bigg[\int_0^{\infty} X^0_t h_x(t,X^{0}_t) dt\bigg] =\mathbb{E}_{(x,y)}\bigg[\int_0^{\infty} e^{-\rho t} X^0_t C'(X^0_t) dt\bigg] \nonumber \\
&& \leq \int_0^{\infty} e^{-\rho t} K_2\Big(\mathbb{E}_{(x,y)}\big[X^0_t\big] + \mathbb{E}_{(x,y)}\big[(X^0_t)^{\gamma}\big]\Big) dt < \infty,
\end{eqnarray}
where the last step is due again to Assumption \ref{ass:rho}. The proof is then completed.
\ep

\vspace{0.4cm}

\emph{Proof of Proposition \ref{prop:valuefinite}}
\vspace{0.15cm}

We start obtaining the two bounds of \eqref{eq:valueOCfinite}.
From \eqref{eq:J}, it is easy to see that $\mathcal{J}_{x,y}(\nu) \geq 0$ for any $\nu \in \mathcal{A}$. Hence $v(x,y) \geq 0$.
On the other hand, being the admissible policy $\nu \equiv 0$ a priori suboptimal, we have by Assumption \ref{ass:c}
\beq
\label{valuefinite1}
v(x,y) \leq \mathbb{E}_{(x,y)}\bigg[\int_0^{\infty} e^{-\rho t} C(X^0_t) dt \bigg] \leq K\bigg[\frac{1}{\rho} + \int_0^{\infty} e^{-\rho t} \mathbb{E}_{(x,y)}\big[(X^0_t)^{\gamma}\big] dt\bigg],
\eeq
where the last step follows by Tonelli's Theorem. By using \eqref{eq:verificoipotesi-3} the claim follows by a simple integration.
The bound $v(x,y)\leq \kappa x$ follows by noticing that the policy ``immediately reduce the debt ratio to $0$" is a priori suboptimal.

The property $v(0,y)=0$ for any $y \in \mathbb{R}$ is due to the fact that $\mathcal{A}(0,y)=\{\nu \equiv 0\}$ and therefore $v(0,y)=\mathcal{J}_{0,y}(0) = 0$.

Convexity of $x\mapsto v(x,y)$ follows by standard arguments employing the convexity of the set of admissible controls, the linearity of \eqref{dynX}, and the convexity of $C(\,\cdot\,)$.
\ep

\vspace{0.4cm}

\emph{Proof of Lemma \ref{lem:limitZeui}}
\vspace{0.15cm}

As for the limit, by nonnegativity of $e^{-\rho t + Z_t}$, $t\geq 0$, we can invoke Fatou's lemma and obtain
\beq
\label{Fatoulimit}
0 \leq \mathbb{E}_{(z,y)}\big[\liminf_{t \uparrow \infty} e^{-\rho t + Z_t}\big] \leq \liminf_{t \uparrow \infty}\mathbb{E}_{(z,y)}\big[e^{-\rho t + Z_t}\big],
\eeq
and the proof is complete if we show that $\liminf_{t \uparrow \infty}\mathbb{E}_{(z,y)}\big[e^{-\rho t+Z_t}\big] = 0$. To this end, notice that by \eqref{def:Z} we can write
\begin{align}
\label{fatou2}
& \mathbb{E}_{(z,y)}\big[e^{-\rho t + Z_t}\Big] = e^z e^{- (\rho - \delta + g)t} \mathbb{E}\Big[e^{-\int_0^t Y^y_s ds}\Big] \nonumber \\
& = e^z \exp\Big\{-(\rho - \delta + g)t - \mathbb{E}\Big[\int_0^t Y^y_s ds\Big] + \frac{1}{2}\Var\Big(\int_0^t Y^y_s ds\Big)\Big\}, 
\end{align}
where we have used that for any given $t\geq 0$ the random variable $\int_0^t Y^y_s ds$ is Gaussian (cf.\ Lemma \ref{lemm:Y}). Then, employing \eqref{av-intY} and \eqref{var-intY} in \eqref{fatou2} one has
$$\mathbb{E}_{(z,y)}\big[e^{-\rho t + Z_t}\big] \leq e^z \exp\Big\{-\Big(\rho - \delta + g + \frac{a}{\theta} - \frac{\sigma^2}{2 \theta^2}\Big)t - \big(y-\frac{a}{\theta}\big)\big(\frac{1-e^{-\theta t}}{\theta}\big)\Big\},$$
which clearly converges to zero as $t\uparrow \infty$ by Assumption \ref{ass:rho}.

To prove \eqref{ui1}, recall \eqref{syst:XY} and \eqref{Markov-set}, and notice that for any $t\geq 0$, $(x,y) \in (0,\infty) \times \mathbb{R}$ and $z \in \mathbb{R}$ one has
$$e^{Z^{z,y}_t} = \frac{e^z}{x} X^{x,y,0}_t \qquad \text{and\,\,then} \qquad \kappa e^{-\rho t + Z^{z,y}_{t}} = \frac{e^z}{x}\kappa e^{-\rho t} X^{x,y,0}_t.$$
Then arguing as in the proof of Proposition \ref{prop:verificoipotesi} one can easily obtain the claim.
\ep

\vspace{0.4cm}

\emph{Proof of Proposition \ref{propertiesb}}
\vspace{0.15cm}

\emph{1.} The first claim follows from the fact that the mapping $y \mapsto \frac{1}{x}u(\ln(x),y)$ (cf.\ \eqref{def:1suxu}) is nonincreasing. This indeed implies that if $(x,y_1) \in \mathcal{I}$, then $(x,y_2) \in \mathcal{I}$ for any $y_2 > y_1$. Also, by \eqref{equivalences} one has
\beq
\label{equivalences2}
\{y \in \mathbb{R}:\, x < b(y)\} = \{y \in \mathbb{R}:\, u(\ln(x),y) - x < 0\},
\eeq
for any given $x>0$. The set on the right-hand side above is open since it is the preimage of an open set via the continuous mapping $y \mapsto u(\ln(x),y) - x$ (cf.\ Proposition \ref{continuityOS}). Hence the set on the left-hand side of \eqref{equivalences2} is open as well and $y \mapsto b(y)$ is therefore lower-semicontinuous. Hence $b(\,\cdot\,)$ is left-continuous since it is nondecreasing.

\emph{2.} This claim follows from \eqref{equivalences} and the fact that $\hat{y}(z) \leq \frac{C'(e^z)}{\kappa} + \delta - g - \rho$ for any $z \in \mathbb{R}$ (cf.\ Proposition \ref{prop:preliminarybdy}).

\emph{3.} For $y^*>-\infty$, the fact that $b(y)=0$ for any $y \leq y^*$ follows by definition of $b$, upon noticing that $\hat{y}(\cdot) \geq y^*$. 
\ep

\vspace{0.4cm}

\emph{Proof of Proposition \ref{prop:1dimOS}}
\vspace{0.15cm}

\emph{(i)}\,\,Suppose $C'(0)>0$. Because for any arbitrary stopping time $\tau$ the mapping $y \mapsto \mathbb{E}[\int_0^{\tau} e^{-\rho t + (\delta - g) t - \int_0^{t} Y^y_s ds}C'(0) dt + \kappa(e^{-\rho \tau + (\delta - g) \tau - \int_0^{\tau} Y^y_s ds} - 1)]$ is continuous, it follows that $y \mapsto w(y)$ is upper semi-continuous (being the infimum of continuous functions). Recall then that $Y$ is positively recurrent and therefore it hits any point of $\mathbb{R}$ in finite time with probability one. Hence denoting by $\mathcal{C}_{w}:=\{y \in \mathbb{R}: w(y)<0\}$ and $\mathcal{S}_{w}:=\{y \in \mathbb{R}: w(y) = 0\}$, if $\mathcal{S}_w \neq \emptyset$ the first entry time of $Y$ into $\mathcal{S}_{w}$ is an optimal stopping time (cf.\ \cite{PeskShir}, Ch.\ 1, Sec.\ 2, Corollary 2.9). Notice that in fact $\mathcal{S}_{w} \neq \emptyset$. Indeed, if not, one would obtain from \eqref{OS-1dim}, by using \eqref{convention}, that for any $y \in \mathbb{R}$
\begin{eqnarray*}
&& 0 > w(y) = \mathbb{E}\bigg[\int_0^{\infty} e^{-\rho t + (\delta - g) t - y(\frac{1 - e^{-\theta t}}{\theta}) - \int_0^t \Xi_s ds}C'(0) dt\bigg] - \kappa, \nonumber 
\end{eqnarray*}
where we have set $\Xi_s := \frac{a}{\theta}(1 - e^{-\theta s}) + \sigma e^{-\theta s}\int_0^s e^{\theta u}dW_u$, $s\geq 0$.
Taking limits as $y \downarrow -\infty$ in the last expected value above, and invoking the monotone convergence theorem, we reach a contradiction and we conclude that $\mathcal{S}_{w} \neq \emptyset$.

From \eqref{OS-1dim} it is easy to be convinced that $y \mapsto w(y)$ is nonincreasing. Then setting $y^*:=\sup\{y \in \mathbb{R}:\, w(y) \geq 0\}$ we have 
$$\mathcal{C}_{w}:=\{y \in \mathbb{R}: y > y^*\} \quad \text{and} \quad \mathcal{S}_{w}:=\{y \in \mathbb{R}: y \leq  y^*\},$$
and $\sigma^*(y):=\inf\{t\geq 0:\, Y^y_t \leq y^*\}$ is optimal for \eqref{OS-1dim}. Also, using \eqref{byparts} in \eqref{OS-1dim} it is clear that $\{y \in \mathbb{R}:\, \delta - g - \rho - y < 0\} \subseteq \mathcal{C}_{w}$, and therefore $\mathcal{S}_{w} \subseteq \{y \in \mathbb{R}:\, y \leq \delta - g - \rho\}$; i.e.\ $y^* \leq \delta - g - \rho$.

By standard arguments (that we skip here in the interest of lenght), it is possible to show that $\limsup_{\varepsilon \downarrow 0} \frac{w(y^*+\varepsilon) - w(y^*)}{\varepsilon} \leq 0 \leq \liminf_{\varepsilon \downarrow 0} \frac{w(y^*+\varepsilon) - w(y^*)}{\varepsilon}$. Hence $w(\,\cdot\,)\in C^1(\mathbb{R})$, and then $w_{yy}(\,\cdot\,)$ is locally bounded at $y^*$. As a consequence, the probabilistic representation
\begin{align}
\label{rappresentationw2}
w(y)=\mathbb{E}\bigg[\int_0^{\infty} e^{-\rho t + (\delta - g) t - \int_0^t Y^y_s ds}\Big(C'(0) + \kappa(\delta - g - \rho - Y^y_t)\Big)\mathds{1}_{\{Y^y_t > y^*\}} dt\bigg]
\end{align}
holds by It\^o-Tanaka's formula.
Since \eqref{rappresentationw2} holds for any $y \in \mathbb{R}$, then by evaluating \eqref{rappresentationw2} for $y = y^*$, one easily finds that $y^*$ solves \eqref{inteq-ystar}. Finally, arguments similar to (but simpler than) those employed in the proof of Theorem \ref{th:integral} show that \eqref{inteq-ystar} admits a unique solution in $(-\infty, \delta-g-\rho)$.

\vspace{0.08cm} 

\emph{(ii)}\,\,If $C'(0)=0$, by \eqref{convention} it is clear that it is never optimal to stop in finite time in \eqref{OS-1dim}. The claim thus follows. 
\ep


\section{Auxiliary Results}
\label{auxiliaryres}
\renewcommand{\theequation}{B-\arabic{equation}}

\begin{lemma}
\label{lemm:Y}
Let $Y$ be the Ornstein-Uhlenbeck process of \eqref{dynY}. Then for any $t\geq 0$ one has
\beq
\label{av-Y}
\mathbb{E}[Y^y_t] = ye^{-\theta t} + \frac{a}{\theta}(1 - e^{-\theta t})
\eeq
and
\beq
\label{var-Y}
\mathbb{E}[(Y^y_t)^2] = \big(ye^{-\theta t} + \frac{a}{\theta}(1 - e^{-\theta t})\big)^2 + \frac{\sigma^2}{2\theta}(1 - e^{-2\theta t}).
\eeq

\noindent Moreover, the integral process $\{\int_0^t Y^y_u du, t \geq 0\}$ is Gaussian and such that
\beq
\label{av-intY}
\mathbb{E}\bigg[\int_0^t Y^y_u du\bigg] = \frac{a}{\theta}t + \left(y-\frac{a}{\theta}\right)\left(\frac{1- e^{-\theta t}}{\theta}\right).
\eeq
and
\beq
\label{var-intY}
\Var\bigg[\int_0^t Y^y_u du\bigg] = - \frac{\sigma^2}{2\theta^3}(1-e^{-\theta t})^2 + \frac{\sigma^2}{\theta^2}\left(t - \left(\frac{1- e^{-\theta t}}{\theta}\right) \right). 
\eeq
\end{lemma}
\begin{proof}
\eqref{av-Y} and \eqref{var-Y} are easily obtained from \eqref{dynY}. On the other hand, for \eqref{av-intY} and \eqref{var-intY} we refer to page 122 of \cite{JYC}.
\end{proof}

\vspace{0.25cm}

\begin{lemma}
\label{inverse}
Recall \eqref{taustargeneral}. One has that $z \mapsto \tau^*(z)$ is a.s.\ decreasing.
\end{lemma}
\begin{proof}
By employing the definitions \eqref{greekzeta} and \eqref{valueOS-general}, it is easy to see that one has a.s.
\begin{align*}
& \{t\geq 0:\,U_t(z) \geq \zeta_t(z)\} \\
& = \Big\{t\geq 0:\,\essinf_{\tau\geq t}\mathbb{E}\Big[\int_t^{\tau} e^{\int_0^s \beta_u(\omega) ds}\, h_x(\omega, t, e^{z + \int_0^s \beta_u(\omega) du}) ds + \kappa_{\tau}(\omega)\, e^{\int_t^{\tau} \beta_u(\omega) du}\Big|\mathcal{F}_t\Big] \geq \kappa_t\Big\}.
\end{align*}
The latter equivalence and the fact that $x \mapsto h_x(t,x)$ is strictly increasing finally yield the claimed monotonicity of $\tau^*(\,\cdot\,)$.
\end{proof}

\vspace{0.25cm}

\begin{lemma}
\label{lemm:limits}
Take $R>0$, let $u$ as in \eqref{valueOS}, and define $\tau_R:=\inf\{t\geq 0:\, |Y_t| \geq R\,\,\text{or}\,\,|Z_t| \geq R\}$ under $\mathbb{P}_{(z,y)}$. Then one has 
\begin{align}
\label{limiting}
& \lim_{T\uparrow \infty} \lim_{R\uparrow \infty}\mathbb{E}_{(z,y)}\bigg[e^{-\rho(\tau_R \wedge T)}u(Z_{\tau_R \wedge T},Y_{\tau_R \wedge T}) + \int_{0}^{\tau_R\wedge T}e^{-\rho s + Z_s}C'(e^{Z_s})\mathds{1}_{\{Y_s > \hat{y}(Z_s)\}} ds\bigg] \nonumber \\
& \hspace{1.7cm}-\mathbb{E}_{(z,y)}\bigg[\int_{0}^{\tau_R \wedge T}e^{-\rho s + Z_s}\kappa \big(\delta-g-\rho-Y_s\big)\mathds{1}_{\{Y_s \leq \hat{y}(Z_s)\}} ds\bigg]\\
&=\mathbb{E}_{(z,y)}\bigg[\int_0^{\infty}e^{-\rho s + Z_s}C'(e^{Z_s})\mathds{1}_{\{Y_s > \hat{y}(Z_s)\}} ds-\int_{0}^{\infty}e^{-\rho s + Z_s}\kappa \big(\delta-g-\rho-Y_s\big)\mathds{1}_{\{Y_s \leq \hat{y}(Z_s)\}} ds\bigg].\nonumber
\end{align}
\end{lemma}
\begin{proof}
Preliminary notice that $\tau_R \wedge T \uparrow T$ when $R\uparrow \infty$. We now analyse the three addends of the expectation in the left-hand side of \eqref{limiting} separately.

\begin{itemize}
\item[(i)] Notice that $0 \leq e^{-\rho(\tau_R \wedge T)}u(Z_{\tau_R \wedge T},Y_{\tau_R \wedge T}) \leq \kappa e^{-\rho(\tau_R \wedge T) + Z_{\tau_R \wedge T}}$ $\mathbb{P}_{(z,y)}$-a.s., because $0 \leq u(z,y) \leq \kappa e^z$ for any $(z,y) \in \mathbb{R}^2$, and that  
$$e^{-\rho(\tau_R \wedge T) + Z_{\tau_R \wedge T}} = e^{-\rho(\tau_R \wedge T)}\Big(e^z + \int_0^{\tau_R \wedge T}e^{Z_s}\big(\delta - g - Y_s\big)ds\Big),$$
by an integration by parts.
Combining these two facts, denoting by $L>0$ a suitable constant independent of $R$ and $T$, and recalling Assumption \ref{ass:rho} we can write
\begin{align}
\label{stimaconv0}
& 0 \leq \mathbb{E}_{(z,y)}\Big[e^{-\rho(\tau_R \wedge T)}u(Z_{\tau_R \wedge T},Y_{\tau_R \wedge T})\Big] \leq \kappa \mathbb{E}_{(z,y)}\bigg[e^{-\rho(\tau_R \wedge T)}\Big(e^z + \int_0^{\tau_R \wedge T}e^{Z_s}\big(\delta - g - Y_s\big)ds\Big)\bigg] \nonumber \\
& \leq \kappa e^z \mathbb{E}_{(z,y)}\Big[e^{-\rho(\tau_R \wedge T)}\Big] + \kappa \mathbb{E}_{(z,y)}\bigg[e^{-(\rho-\rho_o)(\tau_R \wedge T)}\int_0^{\tau_R \wedge T}e^{- \rho_o (\tau_R \wedge T)}|\delta-g-Y_s|\, e^{Z_s} ds\bigg] \nonumber \\
& \leq \kappa e^z \mathbb{E}_{(z,y)}\Big[e^{-\rho(\tau_R \wedge T)}\Big] + \kappa \mathbb{E}_{(z,y)}\bigg[e^{-(\rho-\rho_o)(\tau_R \wedge T)}\int_0^{\tau_R \wedge T}e^{- \rho_o s}|\delta-g-Y_s|\, e^{Z_s} ds\bigg] \\
& \leq \kappa e^z \mathbb{E}_{(z,y)}\Big[e^{-\rho(\tau_R \wedge T)}\Big] + L\,\kappa\,\mathbb{E}_{(z,y)}\Big[e^{- 2 (\rho-\rho_o)(\tau_R \wedge T)}\Big]^{\frac{1}{2}} \mathbb{E}_{(z,y)}\bigg[\int_0^{\infty}e^{-\rho_o s}|\delta-g-Y_s|^4 ds\bigg]^{\frac{1}{4}} \nonumber  \\
& \hspace{1.5cm} \times \mathbb{E}_{(z,y)}\bigg[\int_0^{\infty}e^{-\rho_o s + 4Z_s} ds\bigg]^{\frac{1}{4}}, \nonumber 
\end{align}
where for the last step we have used H\"older inequality with respect to the measure $\mathbb{P}_{(z,y)}$, Jensen inequality with respect to the measure $\frac{1}{\rho_o}e^{-\rho_o s}ds$, and again H\"older inequality, but now with respect to the measure $\mathbb{P}_{(z,y)}\otimes \frac{1}{\rho_o} e^{-\rho_o s}ds$.

From the second equation in \eqref{syst:XY} it is easy to see that 
\begin{equation}
\label{stimaconv1}
\mathbb{E}_{(z,y)}\bigg[\int_0^{\infty}e^{-\rho_o s}|\delta-g-Y_s|^4 ds\bigg]^{\frac{1}{4}} \leq K_1(y),
\end{equation}
for some $0<K_1(y)<\infty$.
On the other hand,
\begin{align}
\label{stimaconv2}
& \mathbb{E}_{(z,y)}\bigg[\int_0^{\infty}e^{-\rho_o s + 4Z_s} ds\bigg] = \int_0^{\infty}e^{-\rho_o s} \mathbb{E}_{(z,y)}\big[e^{4Z_s}\big] ds \nonumber \\
& =e^{4z}\int_0^{\infty}\exp\Big\{-\rho_o s + 4(\delta-g)s - 4\mathbb{E}\Big[\int_0^s Y^y_u du\Big] + 8 \Var\Big[\int_0^s Y^y_u du\Big]\Big\} ds \\
& \leq K_2(z,y), \nonumber 
\end{align}
for some $0 < K_2(z,y) < \infty$, independent of $R$ and $T$. The last inequality above is due the fact that $\rho_o > 4[\delta - g - \frac{a}{\theta} + \frac{2\sigma^2}{\theta^2}] \vee 0$ (cf.\ Assumption \ref{ass:rho}), upon using \eqref{av-intY} and \eqref{var-intY} from Appendix \ref{auxiliaryres} and employing simple estimates.

Thanks to \eqref{stimaconv1} and \eqref{stimaconv2} we can then continue from \eqref{stimaconv0} by writing
\begin{align*}
\label{stimaconv3}
& 0 \leq \mathbb{E}_{(z,y)}\Big[e^{-\rho(\tau_R \wedge T)}u(Z_{\tau_R \wedge T},Y_{\tau_R \wedge T})\Big] \leq \kappa e^z\,\mathbb{E}_{(z,y)}\Big[e^{-\rho(\tau_R \wedge T)}\Big] + K_3(z,y)\,\mathbb{E}_{(z,y)}\Big[e^{-\rho(\tau_R \wedge T)}\Big]^{\frac{1}{2}},
\end{align*}
for some $0<K_3(z,y) < \infty$ independent of $R$ and $T$. Taking now limits as $R \uparrow \infty$, and invoking the dominated convergence theorem we have from the latter
\begin{eqnarray}
\label{stimaconv4}
&0 \leq \limsup_{R\uparrow \infty }\mathbb{E}_{(z,y)}\Big[e^{-\rho(\tau_R \wedge T)}u(Z_{\tau_R \wedge T},Y_{\tau_R \wedge T})\Big] \leq \kappa e^z e^{-\rho T} + K_3(z,y)\,e^{-\frac{\rho}{2} T}.
\end{eqnarray}

\item[(ii)] By the monotone convergence theorem it follows that
\beq
\label{stimaconv5}
\lim_{R\uparrow \infty }\mathbb{E}_{(z,y)}\bigg[\int_{0}^{\tau_R \wedge T}e^{-\rho s + Z_s} C'(e^{Z_s})\mathds{1}_{\{Y_s > \hat{y}(Z_s)\}} ds\bigg] = \mathbb{E}_{(z,y)}\bigg[\int_{0}^{T}e^{-\rho s + Z_s} C'(e^{Z_s}) \mathds{1}_{\{Y_s > \hat{y}(Z_s)\}} ds\bigg].
\eeq

\item[(iii)] Also, writing $\big(\delta-g-\rho-Y_s\big)= \big(\delta-g-\rho-Y_s\big)^+ - \big(\delta-g-\rho-Y_s\big)^{-}$, $s\geq 0$, recalling \eqref{ui1} of Lemma \ref{lem:limitZeui}, and applying the monotone convergence theorem to each of the resulting two integrals we find
\begin{align}
\label{stimaconv6}
& \lim_{R\uparrow \infty }\mathbb{E}_{(z,y)}\bigg[\int_{0}^{\tau_R \wedge T}e^{-\rho s + Z_s}\big(\delta-g-\rho-Y_s\big)\mathds{1}_{\{Y_s \leq \hat{y}(Z_s)\}} ds\bigg] \nonumber \\
&= \mathbb{E}_{(z,y)}\bigg[\int_{0}^{T}e^{-\rho s + Z_s}\big(\delta-g-\rho-Y_s\big)\mathds{1}_{\{Y_s \leq \hat{y}(Z_s)\}} ds\bigg].
\end{align}
\end{itemize}

Then taking limits as $R \uparrow \infty$ in the expectation on the left hand-side of \eqref{limiting} and employing \eqref{stimaconv4}, \eqref{stimaconv5} and \eqref{stimaconv6} yield
\begin{align}
\label{repru3}
& \mathbb{E}_{(z,y)}\bigg[\int_{0}^{T}e^{-\rho s + Z_s} C'(e^{Z_s}) \mathds{1}_{\{Y_s > \hat{y}(Z_s)\}} ds - \int_{0}^{T}e^{-\rho s + Z_s}\kappa \big(\delta-g-\rho-Y_s\big)\mathds{1}_{\{Y_s \leq \hat{y}(Z_s)\}} ds\bigg] \nonumber \\
& \leq u(z,y) \leq \kappa e^z e^{-\rho T} + K_3(z,y)\,e^{-\frac{\rho}{2} T} + \mathbb{E}_{(z,y)}\bigg[\int_{0}^{T}e^{-\rho s + Z_s} C'(e^{Z_s}) \mathds{1}_{\{Y_s > \hat{y}(Z_s)\}} ds\bigg] \\
& \hspace{1.7cm}-\mathbb{E}_{(z,y)}\bigg[\int_{0}^{T}e^{-\rho s + Z_s}\kappa \big(\delta-g-\rho-Y_s\big)\mathds{1}_{\{Y_s \leq \hat{y}(Z_s)\}} ds\bigg].\nonumber
\end{align} 
Finally, taking also limits as $T\uparrow \infty$ in \eqref{repru3}, and arguing as in (ii) and (iii), we conclude that \eqref{limiting} holds true.
\end{proof}

\vspace{0.25cm}

\begin{lemma}
\label{lemm:uleqw}
Let $w$ be defined by \eqref{def:H} and let $u$ be given in terms of representation \eqref{OSrepresent-new}. Then $u \leq w$ on $\mathbb{R}^2$.
\end{lemma}
\begin{proof}
The proof is organised in four steps.\vspace{0.15cm}

\emph{Step 1.} Since by assumption $b$ solves integral equation \eqref{inteq-2}, then $w(z,b(z))=\kappa e^z$, $z \in \mathbb{R}$, and therefore 
\beq
\label{comparison1}
w(z,b(z))=\kappa e^z \geq u(z,b(z)), \quad z \in \mathbb{R}.
\eeq
\vspace{0.25cm}

\emph{Step 2.} Here we show that $w(z,y) = \kappa e^z$ for any $y < b(z)$ and $z \in \mathbb{R}$. This fact clearly implies that
\beq
\label{comparison2}
w(z,y)=\kappa e^z \geq u(z,y), \quad y < b(z),\, z \in \mathbb{R}.
\eeq

Let $z\in \mathbb{R}$ be given and fixed, take $y<b(z)$ and notice that by definition of $H_b$ (cf.\ \eqref{defHb}) one has
$$H_b(Z_s,Y_s) = - \kappa(\delta - g-\rho - Y_s)e^{Z_s}, \quad \forall s \leq \sigma\,\,\,\mathbb{P}_{(z,y)}-a.s.$$
where we have defined $\sigma:=\inf\{t \geq 0:\, Y_t \geq b(Z_t)\}$, $\mathbb{P}_{(z,y)}$-a.s.
Then the martingale property of $\{e^{-\rho t}w(Z_{t},Y_{t})+ \int_{0}^{t} e^{-\rho s}H_{b}(Z_s,Y_s) ds, t\geq 0\}$ (cf.\ \eqref{subharmonicw}) and the optional sampling theorem (cf.\ Ch.\ II, Theorem 3.2 in \cite{RY}) yield
\begin{align}
\label{comparison2-2}
& w(z,y) = \mathbb{E}_{(z,y)}\bigg[\kappa e^{-\rho \sigma + Z_{\sigma}}\mathds{1}_{\{\sigma < n\}} + e^{-\rho n}w(Z_{n},Y_{n})\mathds{1}_{\{\sigma \geq n\}} + \int_{0}^{\sigma \wedge n} e^{-\rho s}H_{b}(Z_s,Y_s) ds\bigg],
\end{align}
for $y < b(z)$, $z \in \mathbb{R}$. Since now $\mathbb{E}_{(z,y)}[e^{-\rho n}w(Z_{n},Y_{n})\mathds{1}_{\{\sigma \geq n\}}] = \mathbb{E}_{(z,y)}[\mathds{1}_{\{\sigma \geq n\}}\int_n^{\infty} e^{-\rho s} H_b(Z_s,Y_s)ds]$ and $H_b(Z_s,Y_s) \in L^1(\mathbb{P}_{(z,y)} \otimes e^{-\rho s}ds)$, we can take limits as $n\uparrow \infty$ in \eqref{comparison2-2} and obtain
\beq
\label{comparison2-3}
w(z,y) = \kappa\,\mathbb{E}_{(z,y)}\bigg[e^{-\rho \sigma + Z_{\sigma}}- \int_{0}^{\sigma} e^{-\rho s + Z_s}\Big(\delta - g - \rho - Y_s\Big) ds\bigg] = \kappa e^z,
\eeq
for $\quad y < b(z)$, $z \in \mathbb{R}$, and where the last equality is due to \eqref{byparts}. Hence, \eqref{comparison2} follows.

\vspace{0.25cm}

\emph{Step 3.} Here we show that 
\beq
\label{comparison3}
w(z,y) \geq u(z,y), \quad y > b(z),\, z \in \mathbb{R}.
\eeq

Let $z\in \mathbb{R}$ be given and fixed, take $y>b(z)$ and consider the stopping time $\tau:=\inf\{t \geq 0:\, Y_t \leq b(Z_t)\}$, $\mathbb{P}_{(z,y)}$-a.s. Then, arguing as in \emph{Step 2} we find
\beq
\label{comparison3-1}
w(z,y) = \mathbb{E}_{(z,y)}\bigg[e^{-\rho \tau + Z_{\tau}} + \int_0^{\tau} e^{-\rho s + Z_s} C'(e^{Z_s}) ds\bigg] \geq u(z,y).
\eeq
\vspace{0.25cm}

\emph{Step 4.} Combining \emph{Step 1}, \emph{Step 2} and \emph{Step 3} we conclude that $w \geq u$ on $\mathbb{R}^2$.
\end{proof}



\begin{thebibliography}{199}



\bibitem{Baker} \textsc{Baker, C.T.H.}\ (1977). \textsl{The Numerical Treatment of Integral Equations}. Clarendon Press, Oxford.

\bibitem{BK}\textsc{Baldursson, F.M., Karatzas, I.}\ $(1997)$. Irreversible Investment and Industry Equilibrium. \textsl{Finance Stoch.} \textbf{1} 69--89.

\bibitem{BankRiedel1} \textsc{Bank, P., Riedel, F.}\ $(2001)$. Optimal Consumption Choice with Intertemporal Substitution. \textsl{Ann.\ Appl.\ Probab.} \textbf{11} 750--788.


\bibitem{BensoussanLions}\textsc{Bensoussan, A., Lions, J.L.}\ (1982). \textsl{Applications of Variational Inequalities to Stochastic Control}. North Holland Publishing Company.

\bibitem{BlanchardFischer}\textsc{Blanchard, O., Fischer, S.}\ $(1989)$. \textsl{Lectures in Macroeconomics}. Cambridge, MA and London: MIT Press.

\bibitem{CadAgui} \textsc{Cadenillas, A., Huam\'an-Aguilar, R.}\ (2016). Explicit Formula for the Optimal Government Debt Ceiling. \textsl{Ann.\ Oper.\ Res.} \textbf{247(2)} pp.\ 415--449.

\bibitem{CadAguiOR} \textsc{Cadenillas, A., Huam\'an-Aguilar, R.}\ (2015). Government Debt Control: Optimal Currency Portfolio and Payments. \textsl{Oper.\ Res.} \textbf{63(5)} pp.\ 1044--1057.

\bibitem{CMZ11} \textsc{Cecchetti, S.G., Mohanty, M.S., Zampolli F.}\ (2011). The Real Effects of Debt. \textsl{Bank for International Settlements}. 

\bibitem{CH94} \textsc{Chiarolla, M.B., Haussmann, U.G.}~(1994). The Free Boundary of the Monotone Follower. \textsl{SIAM J.~Control Optim.}\ \textbf{32(3)} 690--727.


\bibitem{CH00} \textsc{Chiarolla, M.B., Haussmann, U.G.}~(2000). Controlling Inflation: the Infinite Horizon Case. \textsl{Appl.\ Math.\ Optim.}\ \textbf{41} 25--50.

\bibitem{CFR} \textsc{Chiarolla, M.B., Ferrari, G., Riedel, F.}\ $(2013)$. Generalized Kuhn-Tucker Conditions for N-Firm Stochastic Irreversible Investment under Limited Resources. \textsl{SIAM J.~Control Optim.}\ \textbf{51(5)} 3863--3885.

\bibitem{Christensenetal} \textsc{Christensen, S., Crocce, F., Mordecki, E., Salminen, P.}\ (2016). On Optimal Stopping of Multidimensional Diffusions. Preprint on \textbf{arXiv}:1611.00959.

\bibitem{Dayanik} \textsc{Dayanik, S.}\ (2008). Optimal Stopping of Linear Diffusions with Random Discounting.\textsl{Math. Oper.~Res.}\ \textbf{33(3)} 645--661.

\bibitem{DeA2015}\textsc{De Angelis, T.}\ (2015). A Note on the Continuity of Free-Boundaries in Finite-Horizon Optimal Stopping Problems for One Dimensional Diffusions.  \textsl{SIAM J.~Control Optim.}\ \textbf{53(1)} 167--184.

\bibitem{DeAFeFe}\textsc{De Angelis, T., Federico, S., Ferrari, G.}\ (2017). Optimal Boundary Surface for Irreversible Investment with Stochastic Costs. \textsl{Math. Oper.~Res.}\ DOI: 10.1287/moor.2016.0841.


\bibitem{Delves} \textsc{Delves, L.M., Mohamed, J.L.}\ (1985). \textsl{Computational Methods for Integral Equations}. Cambridge University Press.

\bibitem{DP}\textsc{Dixit, A.K., Pindyck, R.S.}\ $(1994)$. \textsl{Investment under Uncertainty}. Princeton University Press. Princeton.

\bibitem{KaratzasElKarouiSkorohod}\textsc{El Karoui, N., Karatzas, I.}\ $(1991)$. A New Approach to the Skorohod Problem and its Applications. \textsl{Stoch.~Stoch.~Rep.} \textbf{34} 57--82.

\bibitem{FedericoPham}\textsc{Federico, S., Pham, H.}\ (2014). Characterization of the Optimal Boundaries in Reversible Investment Problems. \textsl{SIAM J.~Control Optim.}\ \textbf{52(4)} 2180--2223.

\bibitem{Ferrari} \textsc{Ferrari, G.}\ (2015). On an Integral Equation for the Free-Boundary of Stochastic, Irreversible Investment Problems. \textsl{Ann.\ Appl.\ Probab.}\ \textbf{25(1)} 150--176.

\bibitem{Fisher}\textsc{Fisher, I.}\ (1896). Appreciation and Interest. \textsl{Publications of the American Economic Association} \textbf{XI(4)} 331--442.

\bibitem{GT} \textsc{Guo, X., Tomecek, P.}\ (2008). Connections between Singular Control and Optimal Switching. \textsl{SIAM J.~Control Optim.}~\textbf{47(1)} 421--443.

\bibitem{Hackbusch} \textsc{Hackbusch, W.}\ (1994). \textsl{Integral Equations-Theory and Numerical Treatment}. Birkh\"auser.

\bibitem{Hamilton} \textsc{Hamilton, J.}\ (1989). A New Approach to the Analysis of Non-Stationary Time Series and the Business Cycle. \textsl{Econometrica} \textbf{57} 357--384.

\bibitem{HHSZ} \textsc{Hernandez-Hernandez, D., Simon, R.S., Zervos, M.}\ (2015). A Zero-Sum Game Between a Singular Controller and a Discretionary Stopper. \textsl{Ann.\ Appl.\ Probab.}\ \textbf{25(1)} 46--80.


\bibitem{HH} \textsc{Hindy, A., Huang, C.-F.}\ (1993). Optimal Consumption and Portfolio Rules with Durability and Local Substitution. \textsl{Econometrica} \textbf{61} 85--121.

\bibitem{Jacka} \textsc{Jacka, S.D.}\ (1993). Local Times, Optimal Stopping and Semimartingales. \textsl{Ann.\ Probab.}\ \textbf{57} 357--384.

\bibitem{JYC}\textsc{Jeanblanc, M., Yor, M., Chesney, M.}\ $(2009)$. \textsl{Mathematical Methods for Financial Markets}, Springer.

\bibitem{PeskirJohnson}\textsc{Johnson, P., Peskir, G.}\ (2014). Quickest Detection Problems for Bessel Processes. \textsl{Research Report No.\ 25} of the Probab.\ Statist.\ Group Manchester. Forthcoming in \textsl{Ann.\ Appl.\ Probab.}

\bibitem{Karatzas81}\textsc{Karatzas, I.}\ $(1981)$. The Monotone Follower Problem in Stochastic Decision Theory. \textsl{Appl.~Math.~Optim.} \textbf{7} 175--189.

\bibitem{KaratzasShreve84}\textsc{Karatzas, I., Shreve, S.E.}\ $(1984)$. Connections between Optimal Stopping and Singular Stochastic Control I. Monotone Follower Problems. \textsl{SIAM J.~Control Optim.} \textbf{22} 856--877.

\bibitem{KS-MF}\textsc{Karatzas, I., Shreve, S.E.}\ $(1998)$. \textsl{Methods of Mathematical Finance}, Springer.


\bibitem{Lachal}\textsc{Lachal, A.}\ $(1996)$. Quelques Martingales Associées \`a l'Integrále du Processus d'Ornstein-Uhlenbeck. Application \`a l'Etude Despremiers Instants d'Atteinte. \textsl{Stoch.~Stoch.~Rep.} \textbf{58} 285--302.

\bibitem{LeeWu}\textsc{Lee, H.-Y., Wu, J.-L.}\ (2001). Mean Reversion of Inflation Rates: Evidence from 13 OECD Countries. \textsl{J.\ Macroecon.}\ \textbf{23(3)} 477--487.

\bibitem{Liebermann}\textsc{Liebermann, G.M.}\ (2005). \textsl{Second Order Parabolic Differential Equations}. World Scientific.

\bibitem{MenaldiTaksar} \textsc{Menaldi, J.L., Taksar, M.I.}\ (1989). Optimal Correction Problem of a Multidimensional Stochastic System. \textsl{Automatica} \textbf{25(2)} 223--232.

\bibitem{MZ} \textsc{Merhi, A., Zervos, M.}\ (2007). A Model for Reversible Investment Capacity Expansion. \textsl{SIAM J.~Control Optim.}~\textbf{46(3)} 839--876.

\bibitem{PeskShir}\textsc{Peskir, G., Shiryaev, A.}\ (2006). \textsl{Optimal Stopping and Free-Boundary Problems}. Lectures in Mathematics ETH, Birkhauser.

\bibitem{RRR12} \textsc{Reinhart, C.M., Reinhart, V.R., Rogoff, K.S.}\ (2012). Debt Overhangs: Past and Present (No.\ w18015). \textsl{National Bureau of Economic Research}.

\bibitem{RY} \textsc{Revuz, D., Yor, M.}\ (1999). \textsl{Continuous Martingales and Brownian Motion}. Springer-Verlag. Berlin.

\bibitem{RS} \textsc{Riedel, F., Su, X.}\ (2011). On Irreversible Investment. \textsl{Finance Stoch.}~\textbf{15(4)} 607--633.

\bibitem{Shir} \textsc{Shiryaev, A.N.}\ (2008). \textsl{Optimal Stopping Rules}. Springer.

\bibitem{SonerShreve}\textsc{Soner, H.M., Shreve S.E.}~(1989). Regularity of the Value Function for a Two-dimensional Singular Stochastic Control Problem. \textsl{SIAM J.~Control Optim.}\ \textbf{27(4)} 876--907.

\bibitem{Shreve} \textsc{Shreve S.E.}~(1988). \textsl{An Introduction to Singular Stochastic Control}, in {Stochastic Differential Systems, Stochastic Control Theory and Applications}, IMA Vol.\ 10, W.\ Fleming and P.-L.\ Lions, ed.\ Springer-Verlag, New York.


\bibitem{CMW} \textsc{Williams,S.A., Chow, P.-L., Menaldi, J.L.}\ $(1994)$. Regularity of the Free Boundary in Singular Stochastic Control. \textsl{J.~Differential Equations}\ \textbf{111} 175--201.

\bibitem{WK15} \textsc{Woo, J., Kumar, M.S.}\ (2015). Public Debt and Growth. \textsl{Economica} \textbf{82(328)}, 705--739.

\end{thebibliography}
\end{document}